\documentclass[11pt]{amsart}
\usepackage{amssymb,amsmath,epsfig,mathrsfs, enumerate, xparse}
\usepackage[nodisplayskipstretch]{setspace}
\usepackage{graphicx}
\usepackage[normalem]{ulem}
\usepackage{fancyhdr}
\usepackage{esint}
\usepackage{tikz}
\usetikzlibrary{decorations.pathreplacing}
\usepackage{circuitikz}
\usepackage[margin=2.5cm]{geometry}
\usepackage{hyperref}
\hypersetup{
 colorlinks   = true,
 urlcolor     = blue,
 linkcolor    = blue,
 citecolor   = red ,
 bookmarksopen=true
}
\usepackage[noabbrev]{cleveref}

\theoremstyle{plain}
\newtheorem{thrm}{Theorem}[section]
\newtheorem{lemma}[thrm]{Lemma}
\newtheorem{prop}[thrm]{Proposition}
\newtheorem{cor}[thrm]{Corollary}
\newtheorem{rmrk}[thrm]{Remark}
\newtheorem{dfn}[thrm]{Definition}

\allowdisplaybreaks
\begin{document}
\newcommand{\psr}{P^-}
\newcommand{\Rn}{\mathbb R^n}
\newcommand{\Rm}{\mathbb R^m}
\newcommand{\R}{\mathbb R}
\newcommand{\Om}{\Omega}
\newcommand{\eps}{\varepsilon}
\newcommand{\p}{\partial}
\newcommand{\I}{\mathbb I}
\newcommand{\F}{\mathcal F}
\newcommand{\X}{\mathcal X}
\newcommand{\Y}{\mathcal Y}
\newcommand{\W}{\mathcal W}
\newcommand{\la}{\lambda}
\newcommand{\vf}{\varphi}
\newcommand{\B}{\mathcal{B}}

\newcommand{\K}{\mathcal{K}}

\numberwithin{equation}{section}

\title[Parabolic Harnack etc.]{Harnack inequality for degenerate fully nonlinear parabolic equations}

\author{Vedansh Arya}
\address{Department of Mathematics and Statistics\\
University of Jyväskylä\\ Finland}\email[Vedansh Arya]{vedansh.v.arya@jyu.fi}


\author{Vesa Julin}
\address{Department of Mathematics and Statistics\\
University of Jyväskylä\\ Finland}\email[Vesa Julin]{vesa.julin@jyu.fi}

%
%
%
\keywords{}
\subjclass{35K55, 35K65, 35B65}

\begin{abstract}
We consider degenerate fully nonlinear parabolic equations, which generalize the $p$-parabolic equation with $p>2$ to nondivergence form operators. We prove an intrinsic Harnack inequality for nonnegative solutions and a  weak Harnack inequality for nonnegative supersolutions. These results can be seen as the nondivergence form counterparts of the results by DiBenedetto, Gianazza and Vespri \cite{DGV1} and Kuusi \cite{Ku}.        
\end{abstract}
\maketitle
\section{Introduction and the statement of the main results}

In this paper we consider degenerate parabolic equations  of the type
\begin{equation}\label{maineq}
\p_t u  -   F(D^2u, Du, x, t) =0 \qquad \text{in } \, U_T,
\end{equation}
where $U_T = U\times (0,T]$ is  a cylindrical domain with $U   \subset \R^n$  an open set. We assume  that $F$ is uniformly elliptic with respect to the Hessian and degenerate with respect to the gradient. More precisely,  we assume that  there exist constants $0 < \lambda \leq \Lambda$ and $p> 2$ such that 
\begin{equation*}
    \lambda \, |\xi|^{p-2} \text{Tr}(N) \leq F(M+N,\xi,x,t) - F(M,\xi, x,t) \leq \Lambda \, |\xi|^{p-2} \text{Tr}(N)
\end{equation*}
for all symmetric matrices $M$ and $N\geq 0$, and $F(0,\xi, x,t) = 0$  for  every $\xi \in \R^n$ and $(x,t) \in U_T$. The equation \eqref{maineq} is in nondivergence form and thus the solution has to be interpret in viscosity sense. A canonical model for this class is  the $p$-parabolic equation
\[
\p_t u = \text{div}(|Du|^{p-2}Du) = |Du|^{p-2}\text{Tr}\Big( \big( I +(p\!-\!2) \frac{Du}{|Du|} \otimes \frac{Du}{|Du|}\big) D^2 u\Big), 
\]
highlighting the notation $p\!-\!2$ for the positive exponent.

This class of equations shares many features with the porous medium equation; notably, both exhibit \emph{finite speed of propagation} which results from their nonhomogeneous scaling.  This means that, if the support of a nonnegative initial solution is compact, it remains so throughout the evolution.  This phenomenon is explicitly seen  in the \emph{Barenblatt solution} 
\begin{equation}\label{eq:barenblatt}
\varphi(x,t) = \frac{1}{t^{n\alpha}} \left( 1- c \, \Big(\frac{|x|}{t^\alpha} \Big)^{\frac{p}{p-1}} \right)_+^{\frac{p-1}{p-2}}, 
\end{equation}
where $ \alpha = \frac{1}{n(p-2) +p}$ and $c =\frac{p-2}{p} \alpha^{\frac{1}{p-1}}$, which is the fundamental solution of the $p$-parabolic equation.  The $p$-parabolic equation and its generalizations to quasilinear equations in divergence  form  has been studied extensively and we refer to  \cite{Di} and \cite{DGV3} for an introduction to the topic. 

In contrast, our focus is on the equations in nondivergence form, aiming to prove a Harnack inequality for nonnegative solutions of \eqref{maineq}. This inequality is crucial because it implies H\"older continuity of solutions and estimates the speed at which their support expands. In order to state our results, we introduce homothetic cylinders with parameters 
 $\theta >0 $ and $\rho>0$ as  $Q_\rho(\theta) = B_\rho \times (-\theta \rho^p, \theta\rho^p)$, 
\[
Q_\rho^-(\theta) = B_\rho \times (-\theta \rho^p, 0] \quad \text{and} \quad Q_\rho^+(\theta) = B_\rho \times [0, \theta\rho^p),
\]
where $B_\rho$ is the ball with radius $\rho$ centered at the origin.

\emph{Throughout the paper, by a universal constant, we refer to a constant which  depends only  on the  ellipticity constants $\Lambda$ and $\lambda$,  $p$ and the  dimension $n$.}

Our first main result is a weak Harnack inequality for nonnegative viscosity supersolutions. 
\vspace{-1mm}
\begin{thrm}
\label{thm:weak-harnack}
Let $u$ be  a nonnegative viscosity supersolution  of \eqref{maineq} in $U_T$ and let $(x_0,t_0) \in U_T$. There exist  universal constants $\eps>0$, $c>0$ and $C>1$ such that if $u(x_0,t_0)>0$  and the  intrinsic cylinder $(x_0,t_0) + Q_{3\rho}^-(\theta)$,  with
$\theta =\left(\frac{c}{u(x_0,t_0)} \right)^{p-2}$, is in   $U_T$, then  it holds 
\[
\left(\frac{1}{|Q|}\int_{Q} u(x,t)^\eps \right)^{\frac{1}{\eps}} \leq C u(x_0, t_0), 
\]
where $Q=  (x_0, t_0 - \theta \rho^p) + Q_{\rho}^-(\theta)$ . 
\end{thrm}
\vspace{-1mm}
Our second main result is a Harnack inequality for nonnegative viscosity solutions.
\vspace{-1mm}

\begin{thrm} \label{thm:harnack}
Let $u$ be  a nonnegative viscosity solution  of \eqref{maineq} in $U_T$ and let $(x_0,t_0) \in U_T$. There exist  universal constants $ c_1 >0$  and $C>1$  such that if $u(x_0,t_0)>0$ and the  intrinsic cylinder $ (x_0,t_0) + Q_{4\rho}^-(\theta_1)$, with $\theta_1 = \left(\frac{c_1}{u(x_0,t_0)} \right)^{p-2}$, is in $U_T$ then  it holds 
\[
 \sup_{Q^-} u \leq  C u(x_0,t_0), 
\]
where  $Q^-=  (x_0, t_0 - \theta_1 \rho^p) + Q_{\rho}^-(\theta_1)$. 

Moreover, there exists a universal constant $ c_2\geq  c_1$  such that if  the  intrinsic cylinder $ (x_0,t_0 + 2\theta_2 \rho^p) + Q_{4\rho}^-(\theta_2)$, with $\theta_2 = \left(\frac{c_2}{u(x_0,t_0)} \right)^{p-2}$, is in $U_T$,  then  it holds 
\[
 \inf_{ Q^+} u \geq  \frac{1}{C} u(x_0,t_0), 
\]
where  $Q^+=  (x_0, t_0 + \theta_2 \rho^p) + Q_{\rho}^+(\theta_2)$. 

\end{thrm}

The statement of  Theorem \ref{thm:harnack}  differs slightly from the previous  results in the literature (see e.g.  \cite{DGV3}) as the waiting time  $\theta_2$ does not equal to  $\theta_1$. We may always assume that $\theta_1 \leq \theta_2$, but we  give a simple example in Section \ref{subsec:ex}, which shows that the inequality is in general strict.    We stress that this  is not related to the fact that  the equation is in nondivergence form. Indeed, the same example applies also to the divergence form equations. We may obtain $\theta_1 = \theta_2$ by a scaling argument, if  we strengthen the assumptions by requiring that the solution is nonnegative in a large cylinder, where the size depends on the ellipticity constants. 
 Result of this type is proven in \cite{KurS} in a special  case of \eqref{maineq}.

Harnack inequality for the  $p$-parabolic equation is proven by DiBenedetto  \cite{Di88,Di} and for general quasilinear equations in divergence form by  DiBenedetto, Gianazza and Vespri \cite{DGV1, DGV2, DGV3}. At the same time  the weak Harnack inequality for nonnegative supersolution for divergence form equations is proven by Kuusi  \cite{Ku} and we also refer to his PhD thesis \cite{KuThesis} for further discussion on the topic. In \cite{DH} Daskalopoulos and Hamilton  studies the evolution of the support of nonnegative solution  of the porous medium equation as a geometric flow. The result  was later proven for the  $p$-parabolic equation in \cite{DH1}, but the phenomenon does not seem to be completely understood. We also mention a related work  \cite{Av}, where the author studies the finite speed of propagation and its impact on the boundary value problem. 

For linear equations in nondivergence form  the Harnack inequality is due to the celebrated  result by Krylov and Safonov \cite{KS1,KS}. After the works by Trudinger \cite{Tru80} and Caffarelli \cite{Ca} for nonlinear elliptic equations, the Harnack inequality was proven by Wang \cite{W1} for homogenous  fully nonlinear parabolic equations, i.e., the case $p=2$. We refer to \cite{is} for a comprehensive introduction to the topic. 

As we already mentioned, the nonhomogenous  equation \eqref{maineq} is the nondivergence form counterpart to the equation studied in \cite{DGV1, DGV3}. Prior to our work, the Harnack inequality was proven for 
\begin{equation}\label{eq:model}
u_t= |Du|^{p-2} \text{Tr}\Big( \big( I +(q\!-\!2) \frac{Du}{|Du|} \otimes \frac{Du}{|Du|}\big) D^2 u\Big) = |Du|^{p-2} \Delta_q^N u
\end{equation}
in the series of papers \cite{KPS, KurS, PV}, which also include  the singular case $p <2$, where  we  have an elliptic type Harnack for \eqref{eq:model}, i.e., Harnack without the waiting time. In these works, the proof is based on a  cleverly constructed suitable comparison function. We also mention that an intrinsic Harnack inequality for nonhomogeneous fully nonlinear parabolic equation, where the nonhomogeneity is due to a gradient drift term, is proven in \cite{A,AJ}. However, the Harnack inequality for the general equation \eqref{maineq}  remained open.

We recall that by the classical result of Safonov  \cite{Saf},  the optimal regularity for \eqref{maineq} is H\"older continuity. This is proven in \cite{De}  in the elliptic case, where we even have a Harnack inequality  \cite{IS1} (see also \cite{M}),   but for parabolic equations of the form   \eqref{maineq} this was not known. Here we obtain the H\"older continuity as a corollary of Theorem \ref{thm:harnack}.

\begin{cor}
Let $u$ be  a  viscosity solution  of \eqref{maineq} in $U_T$. Then $u$ is locally H\"older continuous in $U_T$ with a universal exponent $\alpha \in (0,1)$
\end{cor}


Higher-order regularity for equations \eqref{maineq} under stronger assumptions has been extensively studied in both the elliptic and parabolic cases. The authors in \cite{IS13} prove the \( C^{1,\alpha} \) regularity for the equation \( |Du|^{p-2}F(D^2u) = f \), with \( p \geq 2 \). We refer to \cite{APPT, ART} for optimal regularity and \cite{BV, BD} for regularity up to the boundary. In the parabolic case, Hölder continuity of the gradient for solutions of the equation \eqref{eq:model} has been proven in \cite{IJS} for \( p > 1 \). We refer to \cite{At, AR} for equations with inhomogeneity \( f \), \cite{LLY} for its extension to fully nonlinear equations, \cite{LLYZ1} for optimal regularity, and to \cite{LLYZ} for results on regularity up to the boundary. For more comprehensive list of references on the topic see  \cite{LLY} .



 \subsection{Outline of the proof}

Similar to divergence form equations, the main challenge in establishing the Harnack and weak Harnack inequalities for equation \eqref{maineq} is due to its nonhomogeneous scaling. The nonhomogeneity  introduces two purely nonlinear phenomena which we have to take into account in the proof. The first  is the finite speed of propagation, which we already mentioned, and the second is the possible finite time blow up of the solution, which we will discuss in Section \ref{subsec:ex}.

From a technical perspective, these issues arise when attempting to apply a standard covering argument (see, for example, \cite{is}) to prove the weak Harnack inequality. The difficulty lies in the fact that the intrinsic cylinders $Q_\rho^-(\theta)$, with $\theta \simeq s^{2-p}$, associated with superlevel sets $\{ u > s \}$ may shrink faster than the level sets themselves decay as $s$ increases. As a result, covering the superlevel sets with very flat cylinders becomes problematic. This difficulty is even more severe in the proof of the full Harnack inequality, since the estimate provided by the weak Harnack inequality alone is insufficient. Specifically, proving Theorem \ref{thm:harnack} requires obtaining an estimate similar to Theorem \ref{thm:weak-harnack} not just within a single cylinder $Q$, but uniformly for all intrinsic cylinders  inside $Q$.

The strategy of the proof is to first establish a so called basic measure estimate, which states that if $u (0,0)\leq 1$ then the measure of the sublevel set $\{ u <4\}$ in $Q_1^-$ is positive (see Lemma \ref{lem:ABP}). This is a general argument and holds also in the singular case $p \in (1,2)$. Due to the nonhomogeneity of the equation, it is not clear how to apply  the classical ABP method  in this setting.  Instead, we adopt the idea of sliding paraboloids from \cite{Sa, Wa}, replacing paraboloids with test functions better suited to our setting. Similar ideas have been employed in \cite{AS} in elliptic homogenization and in \cite{IS1, M} for degenerate elliptic equations.

We call the second fundamental estimate that we need  \emph{propagation of boundedness}, and it states that  if $u(0,0) \leq m_0$ then the sublevel sets $\{u < L_0^km_0\}$ become increasingly dense in space \emph{at every time level}, where   $L_0>1$ and $m_0>0$ are universal constants (see Lemma \ref{lem:propagation-space-2}). This is the most novel part of the proof, and such an estimate, to the best of our knowledge, is new in the context of parabolic equations in nondivergence form. We present this result only for the degenerate case $p > 2$, but a similar argument extends readily to the homogeneous case $p=2$.

Using these two  fundamental estimates, we are then able to quantify the algebraic decay of the level sets  $\{ u>s\}$  in every intrinsic cylinder $(x,t)+ Q_\rho^-(\theta)$, which is the key in the proof of Theorem~\ref{thm:harnack}.

\section{Notations and Preliminaries}\label{s:n}
A point in space-time will be denoted by $(x,t) \in \R^n \times \R$, where $x \in \R^n$ is a point in space. The inner product of $x,y \in \R^n$ is denoted by $x \cdot y$ and $x \otimes y$ denotes the tensor product, which is a rank one $n \times n$ matrix.  Given a real valued function $f : \Omega \subset \R^n \! \times \! \R \to \R $, we denote by 
$Df$, $D^2 f$ the gradient and the Hessian with respect to  space  and $D_{x,t}f$ is the full gradient  in space-time. We denote the partial derivative  in time by $\p_t f$ and  $\p_{x_i} f$ in space in the direction of $x_i$. We denote by $C(\Omega)$ the class of continuous functions and by $C^k(\Omega)$ the class of functions which are $k$-times continuously differentiable.

We  denote Euclidean norm of $x \in \R^n$ by $|x|$. For a given measurable set $A  \subset  \R^{n+1}$, $|A|$ is  the $(n\!+\!1)\text{-dimensional}$  Lebesgue measure of $A$. We denote the ball with radius $\rho$ centered at $x$ by $B_\rho(x)  \subset  \R^n$ and by $B_\rho$ if it is centered at the origin.  Recall that  homothetic cylinders of radius $\rho>0$  and scaling $\theta >0$ are defined as $Q_\rho(\theta) = B_\rho \!\times \!(-\theta\rho^p, \theta\rho^p)$,
\begin{equation} \label{def:cylinders1}
Q_\rho^-(\theta) = B_\rho\! \times \!(-\theta \rho^p, 0] \quad  \text{and} \quad  Q_\rho^+(\theta) = B_\rho\! \times\! [0, \theta \rho^p).
\end{equation}
In the case $\theta =1$ we simply denote $Q_\rho= Q_\rho(1)$,  $Q_\rho^+= Q_\rho^+(1)$ and $Q_\rho^-= Q_\rho^-(1)$. If the cylinder is centered at $(x_0,t_0)$, we denote it by $(x_0,t_0) + Q_\rho^{\pm}(\theta)$, that is, 
\[
(x_0,t_0) + Q_\rho^-(\theta) =  B_\rho(x_0)\! \times \!(t_0 -\theta\rho^p, t_0]  \quad  \text{and} \quad (x_0,t_0) + Q_\rho^+(\theta) =  B_\rho(x_0)\! \times \![t_0, t_0 +\theta\rho^p).
\]

Let us next recall the definition of a viscosity solution. First, we say that the function $\varphi \in C(\Omega)$, defined in a cylindrical domain $\Omega = U \!\times\! (T_1, T_2]$, \emph{touches the function $u :\Omega \to \R$ from below at $(x_0,t_0)$ if}
\begin{equation}\label{eq:below}
\varphi(x,t) \leq u(x,t) \qquad \text{for all } \, x \in U, \,\, t \leq t_0 \quad \text{and } \quad \varphi(x_0,t_0) = u(x_0,t_0).
\end{equation}
Similarly the function $\varphi \in C(\Omega)$ \emph{touches the function $u :\Omega \to \R $ from above  at $(x_0,t_0)$ if}
\begin{equation}\label{eq:above}
\varphi(x,t) \geq u(x,t) \qquad \text{for all } \, x \in U, \,\, t \leq t_0 \quad \text{and } \quad \varphi(x_0,t_0) = u(x_0,t_0).
\end{equation}

\begin{dfn} \label{def:visco}
A lower semicontinuous  function $u :\Omega \to  \R$ is a viscosity supersolution of \eqref{maineq}  in $\Omega = U \!\times\! (T_1, T_2]$ when the following holds: if $\varphi \in C^2(\Omega)$ touches $u$ from below at $(x_0,t_0)\in \Omega$  as defined in \eqref{eq:below} then
\[
\p_t \varphi(x_0,t_0) -   F(D^2 \varphi, D \varphi, x_0, t_0)  \geq 0.
\]

An upper semicontinuous  function $u :\Omega \to  \R$ is a viscosity subsolution of \eqref{maineq} in $\Omega$ when the following holds: if $\varphi \in C^2(\Omega)$ touches $u$ from above at $(x_0,t_0)\in \Omega$  as defined in \eqref{eq:above} then
\[
\p_t \varphi(x_0,t_0) -   F(D^2 \varphi, D \varphi, x_0, t_0)  \leq 0.
\]
Finally $u \in C(\Omega)$ is a viscosity solution of  \eqref{maineq} if it is both viscosity supersolution and subsolution.  
\end{dfn}
We note that instead of assuming \eqref{eq:below} holds globally  in  $\Omega$, we could merely assume that it does it only in a neighborhood of the point $(x_0,t_0)$. We often write merely supersolution instead of viscosity supersolution in order to shorten the terminology. 

We recall the definition of Pucci's extremal operators (for more detail see \cite{cc}). For an $n\times n$ symmetric matrix $M$, the Pucci's extremal operators with ellipticity constants $0<\lambda\leq\Lambda$ are defined as
\begin{equation} \label{pucci}
\begin{split}
    \mathcal{P}_{\lambda,\Lambda}^-(M)=\lambda\sum_{e_i>0}e_i+\Lambda\sum_{e_i<0}e_i \qquad \text{and}   \qquad    \mathcal{P}_{\lambda,\Lambda}^+(M)=\Lambda\sum_{e_i>0}e_i+\lambda\sum_{e_i<0}e_i\notag,
\end{split}
\end{equation}
where $e_i$'s are the eigenvalues of $M$.

Following the ideas of Caffarelli  \cite{Ca}  we  replace the equation \eqref{maineq} by two extremal  inequalities which takes into account the   ellipticity assumption. In other words,  instead of assuming that   $u \in C(U_T)$ is a viscosity solution of \eqref{maineq}, we use the fact that the operator is uniformly elliptic  and  assume merely that $u$ is a viscosity  supersolution of
\begin{equation}\label{eq1}
    \p_t u - |Du|^{p-2}\mathcal{P}_{\lambda,\Lambda}^-(D^2u) \geq 0
\end{equation}
and  a viscosity subsolution of 
\begin{equation}\label{eq2}
 \p_t u -   |Du|^{p-2}\mathcal{P}_{\lambda,\Lambda}^+(D^2u)\leq 0,
\end{equation}
where $\mathcal{P}_{\lambda,\Lambda}^{\pm}$ are the extremal Pucci operators defined above.  By relaxing the equation \eqref{maineq} with the inequalities \eqref{eq1} and \eqref{eq2}  we simplify the structure of the equation by eliminating the dependence on $x$ and $t$.

%

The equations \eqref{eq1} and \eqref{eq2} are nonhomogeneous and we have to take this into account when  scaling the solution. 

\begin{rmrk}\label{rem:scaling}
If $u$ is a viscosity supersolution of  \eqref{eq1} (or viscosity subsolution of \eqref{eq2}) then for $M,r>0$ the function
\begin{equation} \label{eq:scaling}
v(x,t)=\frac{u(rx,r^pM^{-(p-2)}t)}{M}
\end{equation}
is also a viscosity supersolution of  \eqref{eq1} (or viscosity subsolution of \eqref{eq2}) where it is defined. 
\end{rmrk}

\subsection{Example} \label{subsec:ex}

In this short subsection we  construct a simple example which shows that, unlike in the linear case, we may not assume $\theta_1 = \theta_2$ in Theorem \ref{thm:harnack}. Recall that $\theta_2$ is the waiting time in the future, and by recalling the Barenblatt solution \eqref{eq:barenblatt}, we immediately conclude that  $\theta_2$ is bounded from below, i.e., we may not assume it to be arbitrarily small. This is in contrast to the linear case, where one may decrease the waiting time by enlarging the constant $C$. 

Let us then show that by enlarging the ellipticity ratio $\frac{\Lambda}{\lambda}$,  we may decrease the waiting time $\theta_1$ to be arbitrarily small. The example is  one-dimensional in space, and to that aim we fix large $C_0 >1$ and define 
\[
\alpha(t)= (1+C_0t)^{-\tfrac{1}{p-2}}.
\]
Clearly for $t>-\tfrac{1}{C_0}$ it holds   $\alpha'(t) = -\frac{C_0}{p-2}\alpha(t)^{p-1}$. Moreover, it is easy to see that the function $v:\R \to \R$,  $v(x) = |x|^{\tfrac{p}{p-1}}$, satisfies $|v'(x)|^{p-2} v''(x)= \frac{p^{p-1}}{(p-1)^p}$ for $x \neq 0$, where   $v'$ denotes the spatial derivative.
Let $t_k = -\tfrac{1}{C_0} +\frac{1}{k}$ and assume that $k$ is  large enough so that $t_k <0$. Define a sequence of functions $u_k: (-1,1) \times \R \to \R$
as
\begin{equation} \label{eq:example-1}
u_k(x,t)= 
\begin{cases}
\alpha(t)(1- |x|^{\tfrac{p}{p-1}}), &t >t_k,\\
\alpha'(t_k)(t-t_k)+\alpha(t_k)(1-|x|^{\tfrac{p}{p-1}}), \ \ &t \le t_k.\end{cases}
\end{equation}
Then $u_k$ are continuous,  nonnegative, $u_k(0,0)= 1$  and $\lim_{k \to \infty} u_k(0,-\tfrac{1}{C_0}) = \infty$.   By straightforward computation one may verify that for all $k$ the function $u_k$  is a viscosity solution\footnote{Indeed, the function \eqref{eq:example-1} satisfies the equation \eqref{eq:example-2} pointwise for $(x,t)$ with $x \neq 0$ and $t \neq t_k$, and since $x \mapsto - |x|^{\tfrac{p}{p-1}}$ is not twice differentiable at $x = 0$, there are no  $C^2$-regular test function which touches $u_k$ from below at  points $(0,t)$.} of the equation
\begin{equation} \label{eq:example-2}
\partial_t u - a(x,t) |u'|^{p-2} u''= 0,  
\end{equation}
where
\[
a(x,t) = \begin{cases}
\frac{(p-1)^p}{(p-2)p^{p-1}} \big(1- |x|^{\tfrac{p}{p-1}}\big) C_0, &t >t_k,\\
\frac{(p-1)^p}{(p-2)p^{p-1}} C_0, \ \ &t \le t_k.\end{cases}
\]
 Note that the coefficient $a(\cdot, \cdot)$  is bounded from below and above in $(-\tfrac12, \tfrac12)\times \R$.  Hence, we conclude that if we apply  Theorem \ref{thm:harnack} to the equation \eqref{eq:example-2} after translating and with $\rho = \tfrac18$, then   it holds $\theta_1 \leq \tfrac{8^p}{C_0}$, which converges to zero as $C_0 \to \infty$. This example shows that the solutions of the equation \eqref{maineq} may have finite time blow up. For the divergence form equations  this phenomenon  is discussed in \cite{KuThesis}.

We conclude with the remark that the function defined in  \eqref{eq:example-2} is also  a weak solution of the equation 
\[
\partial_t u - \frac{d}{dx}\left( a_0(x,t) |u'|^{p-2} u'\right)= 0, 
\]
where 
\[
a_0(x,t) = \begin{cases}
\frac{(p-1)^{p-1}}{(p-2)p^{p-1}} \big(1-\frac{p-1}{2p-1}|x|^{\tfrac{p}{p-1}} \big)C_0  , &t >t_k,\\
\frac{(p-1)^p}{(p-2)p^{p-1}} C_0, \ \ &t \le t_k.\end{cases}
\]

\section{Basic measure estimate}

The goal of this section is to prove  the basic measure estimate stated in Lemma \ref{lem:ABP}.  We do this by adapting the idea of sliding paraboloids and direct use of area formula  due to Savin \cite{Sa} to our setting. The original argument in   \cite{Sa} is in the  elliptic setting and it has turned out to be useful when dealing with degenerate equations as observed by Imbert and Silvestre \cite{IS1}. In \cite{IS1}, the arguments are based on sliding cusps. Later, Mooney \cite{M} proved the same result using sliding paraboloids.  The parabolic counterpart to \cite{Sa}  is proven by Wang \cite{Wa}.  Since in our case the second order operator in \eqref{maineq} is degenerate,  we need a slight modification of the original argument and replace the sliding paraboloids with test functions, which are better suited to the equation \eqref{maineq}.  In this section we also include  the singular case $p \in (1,2)$, but we stress that this is the only section where we consider the general case $p>1$.

We begin by defining  the family  which we use to replace the sliding paraboloids.  For parameters $(y,s) \in \R^n\times \R = \R^{n+1}$  we define the function  
\begin{equation} \label{def:contactfunction}
\varphi_{(y,s)}(x,t)=-a^{\frac{1}{p-1}} \big(\tfrac{p-1}{p}\big)|x-y|^{\frac{p}{p-1}}+a(t-s),
\end{equation}
where $a>0$ is a fixed constant.  For a given set of parameters $E  \subset \R^{n+1}$ we define the set of \emph{contact points}  in $\Omega = U \!\times\! (T_1, T_2]$  as 
\begin{equation} \label{def:contactset}
\Gamma(E)=\{(x,t) \in \Omega : \; \text{there exists $(y,s) \in E$ such that   $\varphi_{(y,s)}$ touches   $u$ from below at $(x,t)$}\}.
\end{equation}
The idea for considering  functions defined in  \eqref{def:contactfunction}, is  that  for them $|D \varphi_{(y,s)}|^{p-2}\mathcal{P}_{\lambda,\Lambda}^-(D^2 \varphi_{(y,s)})$ is constant.

Given a viscosity supersolution $u$ of \eqref{eq1} in $\Omega = U\!\times\! (T_1,T_2]$ we approximate it by the inf-convolution, i.e., for $(x,t) \in \Omega$ and $\eps>0$ we define $u_\eps : \Omega \to \R$ as
\[
u_{\eps}(x,t) := \inf_{(y,s) \in \Omega} \left( u(y,s) + \frac{1}{2\eps}(|x-y|^2+ (t-s)^2) \right).
\]
It is well known that $u_\eps$ is still a viscosity supersolution of \eqref{eq1} in any  open set  $\Omega'  \subset \!\subset  U \!\times\! (T_1, T_2)$ when $\eps$ is small, and  that $u_\eps$ converges uniformly  to $u$ in $\Omega'$ as  $\eps \to 0$. The advantage is that $u_\eps$ is semiconcave and thus almost everywhere twice differentiable.  We may thus assume in the technical lemmas below that the supersolution $u$ is semiconcave. In particular,   then $u$ satisfies 
\begin{equation}\label{eq:equation-point}
    \p_t u(x,t) - |Du(x,t)|^{p-2}\mathcal{P}_{\lambda,\Lambda}^-(D^2u(x,t)) \geq 0. 
\end{equation}
almost every $(x,t) \in \Omega'$. Moreover,   at the contact points \eqref{def:contactset} u  is differentiable in space.

The differentiability in space at the contact points  $(x,t) \in \Gamma(E)$ implies that  there is a unique  $(y,s) \in E$ such that $\varphi_{(y,s)}$ touches   $u$ from below at $(x,t)$. Because of this, we may  define a map $\Phi :\Gamma(E) \to E$ as 
\[
\Phi(x,t) = (y,s).
\]
We may  write the map $\Phi$ explicitly by using the contact condition, which in particular, implies $D\varphi_{(y,s)}(x,t) = D u(x,t)$ and we write it as  
\begin{equation} \label{eq:contactcond1}
Du(x,t) = D\varphi_{(y,s)}(x,t)=-a^{\frac{1}{p-1}} |x-y|^{\frac{p}{p-1}-2}(x-y).
\end{equation}
In particular, it holds $|Du(x,t)| = a^{\frac{1}{p-1}} |x-y|^{\frac{1}{p-1}}$. Therefore we deduce that at the contact points it holds
\begin{equation} \label{def:contactcond}
\begin{cases}
y &= x + a^{-1}|Du|^{p-2}Du\\
s &= t-a^{-1}u(x,t) -\frac{p-1}{p} a^{\frac{2-p}{p-1}}|x-y|^{\frac{p}{p-1}}=t-a^{-1}u(x,t) -\frac{p-1}{p}a^{-2}|Du|^p.
\end{cases}
\end{equation}
 
\begin{lemma} \label{lem:ABP-1}
Assume $p >1$ and let $u$ be a nonnegative viscosity supersolution of \eqref{eq1} in $\Omega = U \!\times\! (T_1, T_2]$. Let  $E  \subset \R^{n+1}$  be such that for every $(y,s) \in E$ there is a contact point $(x,t) \in \Gamma(E) $  and  $\Gamma(E)  \subset \!\subset  U \!\times\! (T_1, T_2)$, where the contact set $\Gamma(E)$ is defined in  \eqref{def:contactset}. Then it holds 
\[
|\Gamma(E)| \geq c \, |E|, 
\]
for a universal constant $c>0$.
\end{lemma}
\begin{proof}
By the above discussion we may assume that  $u$ is semiconcave and thus  twice differentiable almost everywhere in $\Omega'   \subset \!\subset   U \!\times\! (T_1, T_2)$ and $\Omega'$ contains $\Gamma(E)$. 
The fact that $\varphi_{(y,s)}$ touches $u$ from below at $(x,t)$ implies the equality \eqref{eq:contactcond1} and gives information for the second derivative in space and for the derivative in time a.e. in $\Gamma(E)$. However, we need to be careful as the function $\varphi_{(y,s)}$ is not twice differentiable at $x=y$ when $p>2$, which is the critical point (in space), that is $D\varphi_{(y,s)}(y,t) = 0$.  However, if $Du(x,t) \neq 0$ and  $u$ is twice differentiable at $(x,t)$ then we deduce from the contact condition that 
\begin{equation} \label{eq:contactcond2}
\begin{split}
&\p_t u(x,t) \leq \p_t \varphi_{(y,s)}(x,t)=a\\
&D^2u(x,t) \geq D^2 \varphi_{(y,s)}(x,t) = -a^{\frac{1}{p-1}} |x-y|^{\frac{2-p}{p-1}}\left( I_n - \frac{p-2}{p-1} \frac{x-y}{|x-y|}\otimes\frac{x-y}{|x-y|} \right) ,\\
\end{split}
\end{equation}
where $\otimes$ stands for the tensor product and $I_n$ for the identity matrix. As \eqref{eq:contactcond2} holds only outside the critical points, we need  to deal the critical  points $Du(x,t) = 0$ separately. Recall that if a  contact point is a critical point  $Du(x,t) = 0$, then by \eqref{eq:contactcond1} it holds $x= y$.
We  thus define the sets
\begin{equation} \label{eq:def-Aj}
\Gamma_j=\{(x,t) \in \Gamma(E): |x-y|>1/j\}
\end{equation}
for $j \geq 1$ and let $E_j\subset E$ be the associated set of parameters with $\Phi(\Gamma_j)=E_j$, and 
\[
\Gamma_0=\{(x,t) \in \Gamma(E): |x-y|=0\}
\]
 with  $E_0 \subset E$  is  such that $\Phi(\Gamma_0) = E_0$. Notice that  $E=\cup_{j=0}^{\infty} E_j$. We divide the proof  in two cases.\\
\textbf{Case (i)}: $|E_0|\geq \frac{|E|}{2}.$

This is the easy case. Recall that by definition it holds  $E_0 =\Phi(\Gamma_0)$.  Since $\Gamma_0$  is the set of critical points, $Du(x,t) = 0$,  then by \eqref{def:contactcond} it holds  $\Phi(x,t)=(x,t-a^{-1}u(x,t))$ in $\Gamma_0$. It follows from the semiconcavity that $u$ is Lipschitz continuous and therefore $\Phi$ is Lipschitz in $\Gamma_0$. We may thus use the  area formula to deduce 
\[
|E_0|=|\Phi(\Gamma_0)| \leq \int_{\Gamma_0} |\det D_{x,t}\Phi| \leq \int_{\Gamma_0} |(1-a^{-1}\p_t u)|.
\]  
By the first contact condition \eqref{eq:contactcond2} $1-a^{-1}\p_t u \ge 0.$ Also, as $u$ is a supersolution of \eqref{eq1} and $Du(x,t) =0$ we deduce that  $-\p_t u \leq 0$ almost everywhere in $\Gamma_0$.  Therefore we have 
\[
\frac{|E|}{2} \leq |E_0| \leq |\Gamma_0| \leq |\Gamma(E)|.
\]
Thus we obtain the result in the  Case (i).

\textbf{Case(ii)}: $|\cup_{j=1}^{\infty} E_j|>\frac{|E|}{2}.$

We notice immediately that by the definition of $\Gamma_j$  in \eqref{eq:def-Aj} and by  $\Phi(\Gamma_j) = E_j$ it holds  $E_j \subset E_{j+1}$. Therefore by the monotone convergence theorem $\lim_{j \to \infty} |E_j| > \frac{|E|}{2}$ and we may choose $j$ for which $|E_j|>\frac{|E|}{2}$. The advantage is that in the associated contact set $\Gamma_j$ it holds $|Du(x,t)|\geq c(j) > 0$. Since the argument is long, we divide it in three steps. In the first step we show that the map $\Phi$ is Lipschitz in $\Gamma_j$ and in the second step we calculate the determinant of its Jacobian. In the third step we conclude the proof.  

\textbf{Step 1}: 

In order to use the area formula we  show that   $\Phi$ is Lipschitz  in $\Gamma_j$.  Let $(\hat x, \hat t) \in \Gamma_j$. Then there exists $(y,s) \in E_j$ such that  $\varphi_{(y,s)}$, defined in \eqref{def:contactfunction}, touches $u$ from below at $(\hat x,\hat t)$. Recall that since $(y,s) \in E_j$ then $|D\varphi_{(y,s)}(\hat x, \hat t)|\geq c(j)>0$, which means that $\varphi_{(y,s)}$ is smooth near $(\hat x , \hat t)$. Let us write $\varphi_{(y,s)} = \varphi$ for simplicity. By Taylor expansion it holds  for $t \leq \hat t$ and $x\in U$ 
\[
u(x,t) \geq \varphi(x,t ) \geq \varphi(\hat x,  \hat t) + \p_t \varphi(\hat x, \hat t)(t-\hat t)  +  D \varphi(\hat x, \hat t) \cdot (x- \hat x) - C(j)\big( |x- \hat x|^2 + (t-\hat t)^2\big).
\]
By the contact  condition it  holds $u(\hat x, \hat t) = \varphi(\hat x,\hat t)$, $Du(\hat x, \hat t) = D\varphi(\hat x, \hat t)$ and by the definition of  $\varphi_{(y,s)}$ in \eqref{def:contactfunction} $\p_t \varphi(\hat x, \hat t) = a$. Thus we have by the smoothness of $\varphi$
\begin{equation}\label{eq:Lipschitz-1}
u(x,t) - u(\hat x,\hat t)-   Du(\hat x, \hat t) \cdot (x- \hat x) - a(t- \hat t) \geq -C(j) \big( |x-\hat x|^2 + (t- \hat t)^2\big)
\end{equation}
for  $t\leq \hat t$.  On the other hand  the semiconcavity implies that there is $C>1$ such that the function $(x,t) \mapsto u(x,t) - C(|x|^2 +t^2)$ is concave and thus at the point of differentiability it holds
\begin{equation}\label{eq:Lipschitz-2}
u(x,t) - u(\hat x, \hat t)-Du(\hat x, \hat t) \cdot (x-\hat x) -  a(t-\hat t)    \leq C(|x-\hat x|^2+|t- \hat t|^2)
\end{equation}
for all $(x,t)$. In particular we have 
\begin{equation}\label{eq:Lipschitz-3}
|u(x,\hat t) - u(\hat x, \hat t)-Du(\hat x, \hat t) \cdot (x-\hat x)  | \leq C(j)|x- \hat x|^2
\end{equation}
for all $x \in U$.

Let us next fix $(x_1,t_1), (x_2,t_2) \in \Gamma_j$ with  $t_1 \leq t_2$  and prove
\begin{equation}\label{eq:Lipschitz-4}
| Du(x_2,t_2)-Du(x_1,t_1)| \leq C(j)(|x_2-x_1| + t_2-t_1).
\end{equation}
We denote $r = \sqrt{|x_2-x_1|^2 + (t_2-t_1)^2}$ and  fix $z \in \partial B_r(x_2)$, which choice will be clear later in the proof. We use \eqref{eq:Lipschitz-3} first with the choice $(\hat x, \hat t) =(x_1,t_1) $ and $(x,\hat t) =(x_2,t_1)$ and get
\[
u(x_2,t_1) - u(x_1,t_1)-Du(x_1,t_1) \cdot (x_2-x_1)   \leq C(j)r^2,
\]
 and then with the choice $(\hat x, \hat t) =(x_1,t_1) $ and $(x,t) =(z,t_1)$ and get 
\[
u(x_1,t_1) - u(z,t_1)-Du(x_1,t_1) \cdot (x_1-z)   \leq C(j)r^2.
\]
Summing the estimates  yields
\begin{equation}\label{eq:Lipschitz-5}
u(x_2,t_1) - u(z,t_1)-Du(x_1,t_1) \cdot (x_2-z)   \leq C(j)r^2.
\end{equation}
We then use \eqref{eq:Lipschitz-2} with $(\hat x, \hat t)= (x_2,t_2)$ and $(x,t)= (z,t_1)$ and have 
\[
u(z,t_1) - u(x_2,t_2)-Du(x_2,t_2) \cdot (z-x_2)   - a(t_1-t_2) \leq C(j)r^2.
\]
Finally we use \eqref{eq:Lipschitz-1} with $(\hat x, \hat t)= (x_2,t_2)$ and $(x,t)= (x_2,t_1)$ and have
\[
u(x_2,t_2) - u(x_2,t_1)   - a(t_2-t_1) \leq  C(j)r^2.
\]
Summing the estimates yields
\begin{equation}\label{eq:Lipschitz-6}
u(z,t_1)  - u(x_2,t_1) -Du(x_2,t_2) \cdot (z-x_2)    \leq C(j)r^2.
\end{equation}
We combine  \eqref{eq:Lipschitz-5} with \eqref{eq:Lipschitz-6} and deduce 
\[
\big(Du(x_2,t_2)  -Du(x_1,t_1)\big) \cdot (x_2-z) \leq C(j)r^2.
\]
By choosing $z =x_2-r\frac{Du(x_2,t_2)  -Du(x_1,t_1)}{|Du(x_2,t_2)  -Du(x_1,t_1)|}$ we obtain \eqref{eq:Lipschitz-4} and  deduce that  $\Phi$ is Lipschitz continuous in $\Gamma_j$.

\textbf{Step 2}:

We obtain by  area formula  
\[
|E_j| \le \int_{\Gamma_j} |\det D_{x,t}\Phi|.
\]
Let us next  calculate $\det D_{x,t}\Phi$.  We first  find $D_{x,t} \Phi(x,t)$  by differentiating \eqref{def:contactcond} and to that aim for every $\xi \neq 0$ we define a positive definite symmetric matrix
\begin{equation}\label{def:field-B}
B(\xi) = I_n + (p-2) \frac{\xi}{|\xi|}\otimes \frac{\xi}{|\xi|}.
\end{equation}
Differentiating \eqref{def:contactcond}  gives  for almost all  $(x,t) \in \Gamma_j$
\[
\begin{split}
&D_x y= I_n + a^{-1} |Du|^{p-2}\,  B(Du)D^2 u\\
&\p_t y= a^{-1} |Du|^{p-2}\big(B(Du)\, \p_t Du\big)\\
&D_x s= -a^{-1} Du  - a^{-2} (p-1) |Du|^{p-2} \big(D^2u Du\big) \\
&\p_t s=1-a^{-1}\p_t u -a^{-2} (p-1) |Du|^{p-2} \big(\p_t Du \cdot  Du\big).
\end{split}
\]
 We write the differential of $\Phi$ by the above calculations as
\[
D_{x,t}\Phi =\begin{pmatrix} A & b \\ c^{T} & d_1+d_2 \end{pmatrix},
\]
where $A$ is an  $n\times n$ matrix given by 
\[
A= I_n + a^{-1} |Du|^{p-2}\,  B(Du)D^2 u,
\]
 $b$ and $c$ are vectors given by 
\[
b=  a^{-1} |Du|^{p-2}\big(B(Du)\, \p_t Du\big) \quad \text{and} \quad  c=  -a^{-1} Du  - a^{-2} (p-1) |Du|^{p-2} \big(D^2u Du\big),
\]
 and the numbers $d_1, d_2$ are
\[
 d_1 = 1-a^{-1}\p_t u  \quad \text{and} \quad  d_2 =  -a^{-2} (p-1) |Du|^{p-2} \big(\p_t Du \cdot  Du\big).
\]
 By elementary properties of the determinant we have
\[
\det D_{x,t}\Phi = \det \begin{pmatrix} A & 0 \\ 0 & d_1  \end{pmatrix} +  \det \begin{pmatrix} A & b \\ c^{T} & d_2  \end{pmatrix}= \det M_1 + \det M_2.
\]
We observe that the matrix $M_2 $ above is not full rank as  we obtain its last row as a linear combination of the previous ones. Indeed, by a straightforward  computation we deduce that for $A, b,c$ and $d_2$ from above it holds
\[
-a^{-1} A^T Du = c \qquad \text{and} \qquad -a^{-1} \, b \cdot Du = d_2. 
\]
Therefore if $M_2^j$ denotes the covectors associated with the $j$th row of $M_2$,  then the above  yields
\[
\sum_{j=1}^n a^{-1} \p_{x_j} u \, M_2^j = - M_2^{n+1}.
\]
Hence, $M_2$  is not full rank matrix and $\det M_2 = 0$. Therefore we  obtain 
\begin{equation}\label{eq:determinant-0}
\det D_{x,t}\Phi = \det M_1 = (1-a^{-1}\p_t u(x,t))\det \left(I_n+ a^{-1} |Du|^{p-2}\,  B(Du)D^2 u \right).
\end{equation}

Let us write $B = B(Du)$ defined in \eqref{def:field-B} for short. Since $B$ is symmetric and positive definite we may define its  square root $\sqrt{B}$ and its inverse $\sqrt{B}^{-1}$. In fact,  by direct computation from the definition \eqref{def:field-B} we have
\begin{equation}\label{eq:determinant-root}
\sqrt{B} = \sqrt{B}(\xi) = I_n + q \frac{\xi}{|\xi|} \otimes  \frac{\xi}{|\xi|},
\end{equation}
 where $q >-1$ solves $q^2 + 2q = (p-2)$. Since it holds 
\[
 \det \left(\sqrt{B} \,  \big(\sqrt{B}^{-1}+ a^{-1} |Du|^{p-2}\,  \sqrt{B}\, D^2 u\big) \right) =   \det \left( \big(\sqrt{B}^{-1}+ a^{-1} |Du|^{p-2}\,  \sqrt{B}\, D^2 u\big)\sqrt{B}  \right)\\
\]
we may write the determinant  \eqref{eq:determinant-0}  in a slightly different way as 
\begin{equation}\label{eq:determinant-1}
\det D_{x,t}\Phi =  (1-a^{-1}\p_t u(x,t))\det \left(I_n+ a^{-1} |Du|^{p-2}\,  \sqrt{B}\, D^2 u \, \sqrt{B} \right).
\end{equation}

\textbf{Step 3}:

Let us next show that $\det D_{x,t}\Phi \geq 0$ at the contact points in $\Gamma_j$ where $u$ is twice differentiable. The first contact condition in \eqref{eq:contactcond2} implies $\p_t u \leq a$, which in turn can be written as
\[
1-a^{-1}\p_t u(x,t)) \geq 0.
\]
Let us then show that 
\begin{equation}\label{eq:determinant-2}
I_n+ a^{-1} |Du|^{p-2}\,   \sqrt{B}\, D^2 u \, \sqrt{B} \geq 0,
\end{equation}
where  the matrix $ \sqrt{B} = \sqrt{B}(\xi)$ is defined in  \eqref{eq:determinant-root}. The second contact condition in \eqref{eq:contactcond2} reads as
\[
D^2 u(x,t) \geq   -a^{\frac{1}{p-1}} |x-y|^{\frac{2-p}{p-1}}\left( I_n - \frac{p-2}{p-1} \frac{x-y}{|x-y|}\otimes\frac{x-y}{|x-y|} \right) 
\]
a.e. in $\Gamma_j$. By \eqref{eq:contactcond1} we have $|Du|= \big(a|x-y|\big)^{\frac{1}{p-1}}$ and $\frac{Du}{|Du|} = -\frac{x-y}{|x-y|}$.  Thus we may write the above as 
\begin{equation}\label{eq:ellip1}
 a^{-1} |Du|^{p-2}D^2 u \geq  -I_n + \frac{p-2}{p-1} \frac{Du}{|Du|}\otimes \frac{Du}{|Du|}.
\end{equation}
By recalling the structure  of $\sqrt{B} = \sqrt{B}(Du)$ in  \eqref{eq:determinant-root} we deduce by  direct calculation  that  
\[
\sqrt{B} \left(\frac{Du}{|Du|}\otimes \frac{Du}{|Du|} \right)\sqrt{B}  = (p-1) \, \frac{Du}{|Du|}\otimes \frac{Du}{|Du|}.
\]
Therefore by multiplying   both sides of  \eqref{eq:ellip1} with $\sqrt{B}$ yields
\[
\begin{split}
 a^{-1} |Du|^{p-2} &\big( \sqrt{B}\,   D^2 u \sqrt{B}\,   \big) \\
&\geq -I_n -(p-2) \frac{Du}{|Du|}\otimes \frac{Du}{|Du|} +  \frac{p-2}{p-1}\sqrt{B}\,   \left(\frac{Du}{|Du|}\otimes \frac{Du}{|Du|} \right)\sqrt{B}\,  \\
&= -I_n
\end{split}
\]
and we have \eqref{eq:determinant-2}. 
%

Thus we have shown that $I_n + a^{-1}|Du|^{p-2} \sqrt{B}(Du)\, D^2 u \, \sqrt{B}(Du)$ is symmetric, nonnegative definite matrix. Hence, we have by the arithmetic-geometric  mean inequality
\[
\begin{split}
|E_j| &\leq \int_{\Gamma_j} |\det D_{x,t}\Phi| =  \int_{\Gamma_j} (1-a^{-1}\p_t u(x,t))\det \left(I_n + a^{-1}|Du|^{p-2} \sqrt{B}\, D^2 u \, \sqrt{B} \right)\\
&\leq \frac{1}{(n+1)^{n+1}} \int_{\Gamma_j} \left(n+1 - a^{-1}\p_t u +a^{-1}|Du|^{p-2}\operatorname{Tr}(\sqrt{B}\, D^2 u \, \sqrt{B} ) \right)^{n+1}.
\end{split}
\]
Recall that by \eqref{eq:determinant-2} it holds
\[
a^{-1} |Du|^{p-2}\,  \sqrt{B}\, D^2 u \, \sqrt{B}  \geq - I_n
\]
and that \eqref{eq:contactcond2} implies $\p_t u \leq a$.  Since $u$ is a supersolution of \eqref{eq1}, the inequality  \eqref{eq:equation-point} holds  a.e., which means that  
\[
    |Du|^{p-2}\mathcal{P}_{\lambda,\Lambda}^-(D^2u) \leq  \p_t u .
\]
These together imply
\[
|Du|^{p-2}\,  \sqrt{B}\, D^2 u \, \sqrt{B}   \leq  a \, C \, I_n \quad \text{and} \quad \p_t u \geq - a \, C
\]
for a universal constant $C$. Therefore we finally deduce
\[
\begin{split}
\frac{|E|}{2}< |E_j| &\leq \int_{\Gamma_j} \left(n+1 - a^{-1}\p_t u +a^{-1}|Du|^{p-2}\operatorname{Tr}(\sqrt{B}\, D^2 u \, \sqrt{B} ) \right)^{n+1} \\
&\leq C |\Gamma_j| \leq C|\Gamma(E)|.
\end{split}
\]
This completes the proof.
\end{proof}
\begin{lemma}\label{lem:ABP}
Assume $p >1$.  Let $u$ be a nonnegative viscosity  supersolution of \eqref{eq1} in  $Q_1^- $ with  $u(0,0)\leq 1$. Then there exists $c_0 >0$ such that 
\[
|\{(x,t) \in Q_1^- \, : \, u(x,t) < 4\}| \geq c_0.
\]
\end{lemma}
\begin{proof}
 We define the set of parameters $E \subset \R^{n+1}$ such that  $(y,s) \in E$ if 
\begin{equation} \label{eq:choice:E}
|y| \leq 1/16 \qquad \text{and} \qquad -4\cdot16^{-p} \leq s \leq -2 \cdot 16^{-p}.
\end{equation}

Let $\varphi_{(y,s)}$ be a function as defined in \eqref{def:contactfunction}, i.e., 
\[
\varphi_{(y,s)}(x,t) = -a^{\frac{1}{p-1}} \big(\tfrac{p-1}{p} \big)|x-y|^{\frac{p}{p-1}}+a(t-s),
\]
where we choose 
\[
a =  16^{p}.
\]
With these choices  it holds for  $(y,s) \in E$   that 
\[
\varphi_{(y,s)}(0,0) \geq -   \frac{p-1}{p} + 2 = 1 + \frac{1}{p}  > u(0,0).
\]
On the other hand, when $(x,t) \in Q_1^-$ is such that $|x-y| \geq 1/2$ and $t <0$ then it holds 
\[
\varphi_{(y,s)}(x,t) \leq -8^{\frac{p}{p-1}} \frac{p-1}{p} + 4 < 0 . 
\]
Also when $t < s$, then  it clearly holds $\varphi_{(y,s)}(x,t) < 0$.  Therefore we conclude that for every $(y,s)$ as in \eqref{eq:choice:E}, there is a contact point $(x,t)$, i.e.,   $\varphi_{(y,s)}$ touches $u$ from below at a point $(x,t)$, which satisfies $|x-y| < 1/2$  and $t \geq s$.

 Let $\Gamma(E)$ be the set of contact points, to be more precise 
\[
\Gamma(E)=\{(x,t) \in Q_1^- : \; \text{there exists $(y,s) \in E$ such that   $\varphi_{(y,s)}$ touches   $u$ from below at $(x,t)$}\}.
\]
Then we deduce by the above discussion and by the choice of $y$ in \eqref{eq:choice:E} that if $(x,t) \in \Gamma(E)$ then 
\[
|x| \leq |x-y|+ |y| < \frac12 + \frac{1}{16}  < 1
\]
and  
\[
0> t \geq s \geq - 4 \cdot 16^{-p} > -1.
\]
Then since $\varphi_{(y,s)}(0,0) \geq 1 + \frac{1}{p}$ we have by compactness that $\Gamma(E)  \subset \!\subset B_1 \! \times \!(-1,0)$.  We  deduce by Lemma \ref{lem:ABP-1} that 
\[
|\Gamma(E)| \geq c|E| \geq c_0. 
\]
Note that  at every  contact point $(x,t) \in \Gamma(E)$ it holds 
\[
u(x,t)= \varphi_{(y,s)}(x,t)  < - a s \le 4 .
\]
This completes the proof of the lemma.
\end{proof}

\section{Propagation of boundedness}
In the previous section we proved that if \( u \) is small at one point, then \( u \) is  small on a set of positive measure.  However, this information is given in terms of measure and is  not strong enough to give information on the behavior of $u$ at a specific time level. In this section we prove in Lemma \ref{lem:propagation-space-2} that if $u$ is small at one point, then the sublevel sets $\{ u <s \}$ become increasingly dense in space at every time level.  

To this aim we construct a global barrier function, which mimics the behavior of the Barenblatt solution \eqref{eq:barenblatt}.  Let $q>1$ and consider a smooth decreasing function $g : \R \to [0,\infty)$ such that  
\begin{equation} 
\label{def:func-g}
g(s) = s^{-q}-1 \quad \text{for } \, s \ge \tfrac12  \quad \text{and} \quad g(s) =  2^q \quad \text{for } \, s \in [0,\tfrac{1}{4}]. 
\end{equation}
For $\alpha \in (0,1)$ we consider the function $\psi : \R^n \times (0,\infty) \to \R$ 
\begin{equation} 
\label{eq:barrier}
\psi(x,t) : =  \alpha^{\frac{1}{p-2}}\, t^{-\beta}g\left(\frac{2|x|}{3t^\alpha}\right), \qquad \text{for } \, \beta = \frac{1- \alpha p}{p-2}. 
\end{equation}

We point out that $\psi$ is positive when  $|x| < \tfrac{3}{2} t^\alpha$, and this set is expanding in time.  Note also that the function $\psi$ depends on two parameters $q>1$ and $\alpha \in (0,1)$. We show  that 
$\psi$ is a strict subsolution in its support whenever $\alpha$ is small enough (see the precise formulation in Lemma \ref{lem:barenblatt} below).   We see from the example in Section \ref{subsec:ex} that $\alpha$ needs to be small for $\psi$ to be a viscosity subsolution of \eqref{maineq}.

\begin{lemma}\label{lem:barenblatt}
There exist universal constants $q_0>1$ and $\alpha_0 = \alpha_0(q_0)\in (0,1)$ such that for  $q = q_0$ and for all  $0 < \alpha <\alpha_0$, the function \( \psi \) defined in \eqref{eq:barrier} satisfies   
\[
\partial_t \psi (x,t) - |D \psi|^{p-2} \mathcal{P}^-_{\lambda,\Lambda}(D^2 \psi)(x,t) <0,
\]
for all $(x,t)$ with $|x| < \tfrac{3}{2} t^\alpha.$
\end{lemma}
\begin{proof}
The proof is a  direct computation, where we carefully examine the different ranges determined by the value of the ratio $\tfrac{2|x|}{3t^\alpha}$. We also denote $a=  \alpha^{\frac{1}{p-2}}$ to shorten the notation.   Let us first study  the case  
\[
\frac12 < \frac{2|x|}{3t^\alpha} < 1,
\]
where it holds $g\left(\frac{2|x|}{3t^\alpha}\right) = \left(\frac{2|x|}{3t^\alpha}\right)^{-q} -1$.  Differentiating with respect to time and using the relation  $\alpha=a^{p-2}$  yields
\begin{equation} \label{eq:barneblatt1}
\partial_t \psi = -\beta a\, t^{-\beta-1} g\left(\frac{2|x|}{3t^\alpha}\right) + q\, \alpha \, a  \, t^{-\beta} t^{-1} \left(\frac{2|x|}{3t^\alpha}\right)^{-q} \leq  q\, a^{p-1}  \, t^{-\beta-1}\left(\frac{2|x|}{3t^\alpha}\right)^{-q}.
\end{equation}
We now compute the spatial derivatives of $\psi$, 
\begin{equation} \label{eq:barneblatt2}
D \psi=-a\, q\, t^{-\beta}\left(\frac{2|x|}{3t^{\alpha}}\right)^{-q-1}\frac{2}{3t^{\alpha}}\, \frac{x}{|x|}
\end{equation}
 and 
\[
\begin{split}
D^2 \psi &=- \frac{a\, q\, t^{-\beta} }{|x|}\left(I_n - \frac{x}{|x|} \! \otimes \!\frac{x}{|x|}\right)\left(\frac{2|x|}{3t^{\alpha}}\right)^{\!-q -1} \frac{2}{3t^{\alpha}}\  +a\, q(q+1)\, t^{-\beta} \left(\frac{x}{|x|} \!\otimes \!\frac{x}{|x|}\right) \left(\frac{2|x|}{3t^{\alpha}}\right)^{\!-q-2} \frac{4}{9t^{2 \alpha}}\\
&=a\, q\, t^{-\beta} \left(-I_n +(q+2)\frac{x}{|x|} \otimes \frac{x}{|x|}\right)\left(\frac{2|x|}{3t^{\alpha}}\right)^{-q-2} \frac{4}{9t^{2 \alpha}}.
\end{split}
\] 
The matrix $-I_n + (q+2) \frac{x}{|x|} \otimes \frac{x}{|x|}$ has one  positive eigenvalue which is $q+1$ and $n-1$ negative eigenvalues which  are $-1$. Therefore, when $q$ is sufficiently large, it holds
\[
\mathcal{P}_{\lambda,\Lambda }^-(D^2 \psi)   = a\, q\, t^{-\beta-2\alpha } \frac49  \left( \lambda (q+1) - \Lambda(n-1) \right) \left(\frac{2|x|}{3t^{\alpha}}\right)^{-q-2} \geq  \lambda a\, q \, t^{-\beta-2\alpha} \, \left(\frac{2|x|}{3t^{\alpha}}\right)^{-q-2}.
\]
Therefore, using $\left(\frac{2|x|}{3t^{\alpha}}\right)^{-(q+1)(p-2)} > 1$ (which follows from $\frac{2|x|}{3t^\alpha} < 1$) and the relation $\beta = \frac{1- \alpha p}{p-2},$ we obtain from the above inequality and from \eqref{eq:barneblatt2}
\begin{align*}
-|D \psi|^{p-2} \mathcal{P}_{\lambda,\Lambda }^-(D^2 \psi) & \le -\lambda \, a^{p-1} \, q^{p-1}  t^{-\beta (p-1)} t^{-\alpha p}\left(\frac{2|x|}{3t^{\alpha}}\right)^{-q-2}\frac{2^{p-2}}{3^{p-2}}\\
& \le - \lambda \, a^{p-1} \, q^{p-1}   t^{-\beta -1}  \left(\frac{2|x|}{3t^{\alpha}}\right)^{-q} \frac{2^{p-2}}{3^{p-2}}.
\end{align*}
Combining this with \eqref{eq:barneblatt1} implies
\[
\partial_t \psi -|D \psi|^{p-2} \mathcal{P}_{\lambda,\Lambda }^-(D^2 \psi) \le a^{p-1}  \, t^{-\beta-1}\left(q-\lambda \frac{2^{p-2}}{3^{p-2}} q^{p-1}\right)\ \left(\frac{2|x|}{3t^\alpha}\right)^{-q}.
\]
Choosing $q_0$ sufficiently large, the above yields that for all $a = \alpha^{\frac{1}{p-2}}  > 0$ and $q \ge q_0$, 
\[
\partial_t \psi  -|D \psi|^{p-2} \mathcal{P}_{\lambda,\Lambda }^-(D^2 \psi) < 0.
\]
Hence, the argument for the case $\frac{1}{2} < \frac{2|x|}{3t^\alpha} < 1$ is now complete. Note that here we have fixed $q_0$, but not $\alpha_0$. 

Next we proceed to  the  case 
\[
\frac14 \leq \frac{2|x|}{3t^\alpha} \leq \frac12.
\]
Differentiating in time yields
\[
\partial_t \psi  = -\beta a\, t^{-\beta-1} g\left(\frac{2|x|}{3t^\alpha}\right) - a \alpha \, t^{-\beta-1}  \frac{2|x|}{3t^\alpha}g'\left(\frac{2|x|}{3t^\alpha}\right).
\]
Notice that when $\alpha \leq \frac{1}{2p}$, we have from the choice of $\beta$ in \eqref{eq:barrier}  that $\beta \ge \frac{1}{2(p-2)}$.  Since $g$ is decreasing, it holds $g\left(\frac{2|x|}{3t^\alpha}\right) \geq g\left(\frac12\right) = 2^q -1$.
Therefore since $g$ is smooth and $\alpha = a^{p-2}$ we have
\[
\partial_t \psi \leq  (-ca+Ca^{p-1}) t^{-\beta-1},
\]
for constants $0 < c < 1$ and $C > 1$, which are independent of $\alpha$ when  $\alpha \leq \frac{1}{2p}$.  

We now estimate the spatial derivative of $\psi$ as   
\[
|D \psi|  \le C a t^{-\beta-\alpha} 
\]
 and 
\[
|D^2 \psi| \leq C a \left(  t^{-\beta-2\alpha}  + \frac{ t^{-\beta-\alpha} }{|x|}\right).
\]
Write $|x|^{-1}=\frac23 t^{-\alpha}\left(\frac{2|x|}{3t^\alpha}\right)^{-1},$ use $\frac14 \le \frac{2|x|}{3t^\alpha} \leq \frac12,$ and $\beta = \frac{1- \alpha p}{p-2},$ to find
\[
 |D \psi|^{p-2} |\mathcal{P}_{\lambda,\Lambda }^-(D^2 \psi)| \le C a^{p-1}t^{-\beta (p-1)}t^{-\alpha p}=C a^{p-1}t^{-\beta -1}.
\]
In conclusion,  we obtain
\[
\partial_t \psi- |D \psi|^{p-2} \mathcal{P}_{\lambda,\Lambda }^-(D^2 \psi) \le (-ca+C a^{p-1})t^{-\beta -1}.
\]
Since $p>2$ (or equivalently $p-1>1$), we can choose $\alpha_0>0$ small enough such that for all  $a = \alpha^{\frac{1}{p-2}} \leq \alpha_0^{\frac{1}{p-2}}<\frac1{2p}$ and it holds  
\[
-ca+C a^{p-1} < 0.
\]
 Consequently, for $q=q_0$ and $0<\alpha<\alpha_0$, we have
\[
\partial_t \psi- |D \psi|^{p-2} \mathcal{P}_{\lambda,\Lambda }^-(D^2 \psi) < 0.
\]
This completes the proof of the claim for the case $\frac14 \le  \frac{2|x|}{3t^\alpha} \leq \frac12.$

We are left with the case 
\[
\frac{2|x|}{3t^\alpha} < \frac{1}{4}.
\]
 In this case the argument is trivial as here  $\psi(x,t)=at^{-\beta}2^q,$ and therefore
\[
\partial_t \psi -|D \psi|^{p-2} \mathcal{P}_{\lambda,\Lambda }^-(D^2 \psi)=-a \beta t^{-\beta -1}2^q < 0.
\]
This completes the proof of the lemma.
\end{proof}

The following lemma is an important building block in the proofs of the main theorems. We see from the example in Section \ref{subsec:ex}  that if $u(0,0)=1$, then the intrinsic cylinder $Q$, with $(x_0,t_0)= (0,0)$, in the statement of Theorem \ref{thm:weak-harnack} may be very flat. We choose to assume that the value of the function at the reference point is small, i.e., $u(0,0)\leq m_0$ in order to have the intrinsic cylinder to be of size one.  

 \begin{lemma}
\label{lem:barrier}
There exist universal constants $0<m_0<1$ and $L_0>1$ such that if  $u :Q_3^- \to \R$ is a nonnegative viscosity supersolution of \eqref{eq1} such that   $u(0,0) \leq  m_0,$ then for any $(x_0,t_0) \in B_{4/3}\! \times \!(-3, -\tfrac12]$  there exists $\bar x \in B_{1/32}(x_0)$ such that 
\[
u(\bar x,t_0) \leq L_0m_0.
\]
 \end{lemma}
 \begin{proof} 
We recall the definition of $\psi$ in \eqref{eq:barrier}
\[
\psi(x,t) : =  \alpha^{\frac{1}{p-2}}\, t^{-\beta}g\left(\frac{2|x|}{3t^\alpha}\right) \qquad \text{for } \, \beta = \frac{1- \alpha p}{p-2}, 
\]
where $g$ is defined in \eqref{def:func-g}, with  $q=q_0$ and $\alpha \in (0,\alpha_0)$, where $q_0$ and $\alpha_0$ are from Lemma \ref{lem:barenblatt}.  Fix \( (x_0,t_0)  \in B_{4/3}\! \times \!(-3, -\tfrac12] \) and  define the domain
\[
P = \left\{ (x, t) : |x-x_0| < \tfrac{3}{2} (t-t_0+b)^{\alpha} \quad \text{and} \quad t_0 < t \le 0 \right\},
\]
where $b \in (0,1)$ and which choice will be clear later, and $\alpha$ is small. We claim that $(0,0) \in P$. Indeed, it follows from $x_0 \in B_{4/3}$, $t_0 \leq -\frac12$,  $b\geq 0$  and by choosing $\alpha$ small enough that  
\begin{equation}\label{eq:barrier-1}
|x_0| < \frac{4}{3}\leq \frac{3}{2} (-t_0 + b)^{\alpha} \cdot \frac{17}{18}.
\end{equation}
Hence $(0,0) \in P$. Let us also show that $P \subset Q_3^-$.  By choosing $\alpha$ small and arguing as in  \eqref{eq:barrier-1}, we conclude  that  for all  $(x,t) \in P$ it holds  $|x-x_0|< \frac53$. Therefore, since $|x_0|\leq \frac43$ we conclude that $x \in B_3$. It is then obvious that $(x,t) \in Q_3^- $  as $t \in (-3,0]$.

Let us denote 
\[
\psi_{(x_0,t_0)}(x,t) = \psi(x - x_0, t - t_0 + b). 
\]
Recall that  $(0,0) \in P$ and that we assume  $u(0,0) \leq m_0$.
Now, by choosing $m_0 =  4^{-\beta-1} \alpha^{\frac{1}{p-2}} g\left(\tfrac{17}{18}\right) > 0$,  using the fact that $g$ is  decreasing,  $-t_0 \leq 3,  b \leq 1$ and \eqref{eq:barrier-1} we have
\begin{equation}\label{eq:barr_min}
\begin{split}
\min_{\bar{P}}(u - \psi_{(x_0,t_0)}) 
&\le m_0 - \alpha^{\frac{1}{p-2}}\, (-t_0+b)^{-\beta} g\left(\frac{2|x_0|}{3(-t_0 + b)^\alpha}\right) \\
&\le m_0 - 4^{-\beta}\alpha^{\frac{1}{p-2}}\, g\left(\tfrac{17}{18}\right) < 0.
\end{split}
\end{equation}

Define $(x_m,t_m)$  to be a point where the function $(u- \psi_{(x_0,t_0)})$ attains its minimum in $\bar P$. For  $(x,t)$ with $|x-x_0| = \tfrac{3}{2} (t-t_0+b)^{\alpha},$ we have 
\[
\psi_{(x_0,t_0)}(x,t)=\psi(x-x_0,t-t_0+b)= \alpha^{\frac{1}{p-2}}(t-t_0+b)^{-\beta}g(1)=0.
\]
Recall that  $u \ge 0$. Therefore  
we conclude that  the point $(x_m,t_m)$ does not belong to the part of the boundary of $P$  where $|x-x_0| = \tfrac{3}{2} (t-t_0+b)^{\alpha}.$

Let us next show  that $(x_m,t_m)$ is not in $P$. We argue by contradiction and assume    $(x_m,t_m) \in P$, i.e., 
\[
(u- \psi_{(x_0,t_0)})(x_m,t_m) \leq (u- \psi_{(x_0,t_0)})(x,t) \qquad \text{for all } \, (x,t) \in P.
\]
But this means that the function $\varphi(x,t) = \psi_{(x_0,t_0)}(x,t) + u(x_m,t_m)- \psi_{(x_0,t_0)}(x_m,t_m)$ touches $u$ from below in $P$. Since $u$ is a viscosity supersolution of \eqref{eq1}, we have 
\begin{align*}
0 & \leq \p_t \varphi(x_m,t_m) - |D\varphi|^{p-2}\mathcal{P}_{\lambda,\Lambda}^-(D^2\varphi)(x_m,t_m)\\ 
&= \p_t \psi(x_m-x_0,t_m-t_0+b) - |D\psi|^{p-2}\mathcal{P}_{\lambda,\Lambda}^-(D^2\psi)(x_m-x_0,t_m-t_0+b),
\end{align*}
which contradicts Lemma \ref{lem:barenblatt}  as $(x_m,t_m) \in P$. 

We deduce that  $u-\psi$ attains its minimum on the part of the boundary of $P$ where $$t=t_0 \quad \text{and} \quad |x-x_0| < \frac{3}{2} b^{\alpha}.$$
By choosing $b \in (0,1)$ such that $ \frac{3}{2} b^{\alpha} < \frac{1}{32}$ we have  $x_m \in B_{1/32}(x_0).$ Now, using \eqref{eq:barr_min} and the fact that \( g \) is a decreasing function, we find  
\[
u(x_m,t_0) \le \psi(x_m - x_0, b) =  \alpha^{\frac{1}{p-2}} b^{-\beta} g\left(\frac{2|x_m - x_0|}{3b^{\alpha}} \right) \le \alpha^{\frac{1}{p-2}} b^{-\beta} g(0) = \alpha^{\frac{1}{p-2}}  b^{-\beta} 2^{q_0}.
\]
Thus, the claim follows by choosing \( L_0 =  \frac{\alpha^{\frac{1}{p-2}} b^{-\beta} 2^{q_0}}{m_0} \). 

\end{proof}

The previous lemma is  enough  for the proof of the weak Harnack inequality, but we need to iterate it in order to obtain the propagation of boundedness, which is crucial in the proof of Theorem \ref{thm:harnack}.

\begin{lemma}
\label{lem:propagation-space-2}
Let $0<m_0<1$ and $L_0>1$  be as in the previous lemma and assume  $u: Q_3^-\to \R $ is a nonnegative viscosity  supersolution of \eqref{eq1} with $u(0,0) \leq  m_0$. Then for any $k \in \mathbb{N}$ and  for any $(x_0,t_0) \in  B_{4/3} \!\times\! (-3, -1 + 32^{-kp} ]$ there is $\bar x \in B_{32^{-k}}(x_0)$ such that 
\[
u(\bar x,t_0) \leq  L_0^k m_0.
\]
 \end{lemma}
  
\begin{proof}
We argue by induction on $k$ and note that by Lemma \ref{lem:barrier}  the claim holds for $k =1$. Fix $k$, assume the claim holds for $k$, and assume $(x_0,t_0) \in B_{4/3} \!\times\! (-3, -1 + 32^{-(k+1)p} ].$   Denote $r_k =32^{-k} $ and $\theta_k=L_0^{-k(p-2)}.$ We use  the induction hypothesis to the point $(x_0,t_1)$, where $t_1 = t_0 + r_k^p  \frac{\theta_k}{2}$ and deduce that there is  $x_1$ such that  $|x_1-x_0|< 32^{-k}$ and  
\[
u(x_1,t_1)\leq  L_0^{k}m_0. 
\]
Consider the function $v:Q_3^- \to \R$
\[
v(y,s) = \frac{ u(r_ky+x_1,r_k^p \theta_k s +t_1) }{L_0^{k}} . 
\]
Then by Remark \ref{rem:scaling} $v$ is a nonnegative supersolution of \eqref{eq1} with  $v(0,0)\leq m_0$. Denote $z = \frac{(x_0-x_1)}{r_k}$ and notice that by the assumption $|x_1-x_0|< 32^{-k}$ it holds $|z|< 1$. By Lemma \ref{lem:barrier}, there is $\bar z \in B_{1/32}(z)$ such that  
\[
v(\bar z,-\tfrac12) \leq L_0m_0. 
\]  
Writing in terms of $u$, we have 
\[
u(r_k\bar z+x_1,-r_k^p  \tfrac{\theta_k}{2} +t_1 ) \leq  L_0^{k+1}m_0. 
\]
The definition of $t_1$ immediately yields $-r_k^p   \tfrac{\theta_k}{2} +t_1=t_0.$ Using $z = \frac{(x_0-x_1)}{r_k}$, we find 
$$r_k\bar z+x_1=r_k(\bar z-z)+r_k z +x_1=r_k(\bar z-z)+x_0.$$ Therefore,  using $|\bar z- z|< \frac1{32},$ we get 
\[
|r_k\bar z+x_1-x_0| \le 32^{-(k+1)}
\]
and the claim follows for $k+1$ by choosing $\bar x =r_k\bar z+x_1 $. 

\end{proof}

We also need a version of Lemma \ref{lem:propagation-space-2} with a different scaling. 
\begin{lemma}
\label{lem:propagation-space-3}
Let $0<m_0<1$ and $L_0>1$  be as in the previous lemma,  assume  $u: Q_3^-\to \R  $ is a nonnegative viscosity supersolution of \eqref{eq1} with $u(0,0) \leq  m_0$ and fix $\nu >1$. Then there are $\rho_0 = \rho_0(\nu) \in (0,1)$  and $k_0 \in \mathbb{N} $ such that    for any $(x_0,t_0) \in  B_{4/3} \!\times\! (-3, -1]$ and $k \geq k_0$ there is $\bar x \in B_{\rho_0^{k}}(x_0)$ such that 
\[
u(\bar x,t_0) \leq   \nu^km_0.
\]
 \end{lemma}
 \begin{proof}
 If $\nu \ge L_0$ then we are done by Lemma \ref{lem:propagation-space-2} and taking $\rho_0 =\tfrac1{32}.$ If $\nu <L_0,$ choose $k_0$ such that $L_0 \le \nu^{k_0} <L_0^2.$ Fix $k \ge k_0,$ and let   $\tilde k$ be  such that $L_0^{\tilde k} \le \nu^k <L_0^{\tilde k+1}.$ By Lemma \ref{lem:propagation-space-2}, there is $\bar x \in B_{32^{-\tilde k}}(x_0)$ such that 
\[
u(\bar x,t_0) \le L_0^{\tilde k}m_0.
\]
 We  choose $\rho_0(\nu) = \nu^{-\gamma}$ for a small $\gamma >0$. Then it holds by the choice of $\tilde k$ that 
\[
 \rho_0^k = \big(\nu^{-k}\big)^\gamma >  \big(L_0^{-\tilde k -1}\big)^\gamma. 
\]
By choosing $\gamma$ small enough  we have $\big(L_0^{-\tilde k -1}\big)^\gamma  \ge 32^{-\tilde k}$, and thus  $\rho_0^k > 32^{-\tilde k}.$ 
Hence, $B_{32^{-\tilde k}}(x_0) \subset B_{\rho_0^k}(x_0),$ and  therefore $\bar x \in  B_{\rho_0^k}(x_0)$ with 
\[
u(\bar x,t_0) \le L_0^{\tilde k}m_0 \le  \nu^km_0,
\]
 where the last inequality follows from the choice of $\tilde k.$ This completes the proof of the lemma.
 \end{proof}
\section{Weak Harnack}
In this section we prove Theorem \ref{thm:weak-harnack}. In order to  do this we have to prove that  the measure of the superlevel sets decay algebraically. The proof is via  induction argument, where we use a delicate covering argument to 
estimate the measure of the superlevel set $\{ u > L^k m_0 \}$. To this aim we define the following paraboloid type sets, which  can be seen as a degenerate version of the ones introduced in \cite{Wa}. 
 For a fixed point $(x_0,t_0) \in \R^{n+1}$ and positive numbers $\theta$ and $r$ we define the paraboloid type set   
\begin{equation} \label{def:B-a}
\B^{\theta}_r(x_0,t_0):=\{(x,t)  \in \R^{n+1} : \theta |x-x_0|^p \leq t-t_0 \leq r\}
\end{equation}
and denote $\B_r(x_0,t_0)$ when $\theta = 1$. We remark that we introduce  $\theta$  in the definition \eqref{def:B-a}  in order to take into acount the intrinsic scaling of the equation. Notice that  the set $\B^{\theta}_r(x_0,t_0)$ becomes flatter as the value of  $\theta$ decreases. 




The definition of  $\B^{\theta}_r$ in  \eqref{def:B-a} takes into account the nonhomogeneous scaling of the equation.   For Theorem \ref{thm:weak-harnack} we need a quantitative  estimate  of  the propagation of  boundedness in terms of measure inside the set  $\B^{\theta}_r \cap Q$. We do this  using  the basic measure estimate from  Lemma \ref{lem:ABP}  and Lemma \ref{lem:barrier}.

\begin{lemma}\label{prop:bmu}
Let \( 0 < m_0 <1 \) and \( L_0 > 1 \) be as in Lemma~\ref{lem:barrier}. Suppose \( x_0 \in B_1 \) and \( 0 < \rho \leq 1 \). Assume that \( u : Q_3^- \to \mathbb{R} \) is a nonnegative viscosity supersolution of \eqref{eq1} with
\[
\inf_{B_\rho(x_0)   \cap B_1} u(\cdot, 0) \leq m_0.
\]
Then, there exists a universal constant \( c_1 > 0 \) such that
\[
\left| \left\{ (x, t) \in \mathcal{B}_{\rho^p}(x_0,  - \rho^p) \, : \, x \in B_1 \ \text{ and } \ u(x, t) < 4L_0m_0  \right\} \right| \geq c_1 \rho^{n + p}.
\]
\end{lemma}
\begin{proof}
Our first aim is to find a point $(z_1,t_1)$ which is well inside the set  $ \mathcal{B}_{\rho^p}(x_0,  - \rho^p)$ and  $u(z_1,t_1) \leq L_0 m_0$. We need a few intermediate steps to find it.   The result then follows by applying  the basic measure estimate from  Lemma~\ref{lem:ABP} translated to $(z_1,t_1)$.  

 First, let $x_m \in   \bar{B}_{\rho}(x_0) \cap \bar{B}_1 $ be a point where the infimum 
$ \inf_{B_\rho(x_0) \cap B_1} u(\cdot, 0)$
is attained, i.e., $u(x_m, 0) \leq m_0$. Consider the point
\[
z = \left( \frac{1}{2} - \frac{\rho}{16} \right)(x_m + x_0).
\]
Then we have $|z| < 1 - \tfrac{\rho}{8}$. Moreover, we claim that 
\begin{equation}\label{eq:bmu1}
z \in B_{5\rho/8}(x_0) \qquad \text{and} \qquad z \in B_{5\rho/8}(x_m).
\end{equation}
Indeed, we may estimate
\[
|z - x_0| = \left| \frac{1}{2}(x_m - x_0) - \frac{\rho}{16}(x_0 + x_m) \right| 
\leq \frac{1}{2}|x_m - x_0| + \frac{\rho}{16}|x_0 + x_m|.
\]
Since $|x_m - x_0| \leq \rho$ and $|x_0 + x_m| < 2$, it follows that $|z - x_0| < \frac{5\rho}{8}$.
Similarly we have $z \in B_{5\rho/8}(x_m)$ by symmetry.

 Consider the function $v: Q_3^- \to \R$ defined as
\[
v(y, \tau) = u(\tfrac{\rho}{2} y + x_m, \tfrac{\rho^p}{2^p} \tau).
\]
Then $v$ is a nonnegative supersolution of \eqref{eq1} with $v(0,0) \le m_0.$ Using the second estimate in \eqref{eq:bmu1}, we conclude $\tfrac{2}{\rho}|z-x_m| \le \tfrac54<\tfrac43$. We may thus  apply  Lemma \ref{lem:barrier} to deduce that  there exists a point $y_m \in B_{1/32}\big(\tfrac{2}{\rho}(z-x_m)\big)$ such that 
\[
v(y_m,-\tfrac12) \le L_0m_0.
\]
Equivalently, in the original coordinates, we have
\[
u(\tfrac{\rho}{2} y_m + x_m, -\tfrac{\rho^p}{2^{p+1}} ) \le L_0m_0.
\]
Finally we choose  $(z_1,t_1) = (\tfrac{\rho}{2} y_m + x_m, -\tfrac{\rho^p}{2^{p+1}})$ and claim  that it holds 
\begin{equation}\label{eq:bmu2}
(z_1, t_1) +Q^-_{\rho/32}\subset \mathcal{B}_{\rho^p}(x_0,  - \rho^p).
\end{equation}

To this aim we fix  $(x,t) \in (z_1, t_1) + Q^-_{\rho/32}$. We first estimate  
\[
|x-x_0| \le |x-z_1| +|z_1-x_0| < \frac{\rho}{32} +|z_1-x_0|. 
\]
Using $y_m \in B_{1/32}\big(\tfrac{2}{\rho}(z-x_m)\big)$ and $z \in B_{5\rho/8}(x_0)$ provided by \eqref{eq:bmu1}, we obtain
\[
\begin{split}
 |z_1-x_0| &= \left|\frac{\rho}{2} y_m + x_m -x_0\right| = \left|\frac{\rho}{2}\left(y_m-\frac{2}{\rho}(z-x_m)\right)+z-x_0\right|\\
 &< \frac{\rho}{64} + \frac{5\rho}{8} < \frac{3\rho}{4}.
\end{split}
\]
In conclusion we have
\begin{equation}\label{eq:bmu3}
|x-x_0| <  \frac{\rho}{32} + \frac{3\rho}{4}<\frac{4}{5}\rho.
\end{equation}
Let us then estimate the time variable of the point $(x,t) \in (z_1, t_1) + Q^-_{\rho/32}$. Using  $t_1 = -\tfrac{\rho^p}{2^{p+1}} $ and $p>2$ we have 
\[
t+\rho^p = (t-t_1)+t_1+ \rho^p  \ge -\frac{\rho^p}{32^p}-\frac{\rho^p}{2^{p+1}} +\rho^p > \frac{4}{5}\rho^p.
\]
Moreover it holds $t\leq0$. Therefore, we conclude from \eqref{eq:bmu3} that 
\[
|x-x_0|^p < \left(\frac{4}{5}\rho\right)^p < \frac{4}{5}\rho^p  < t +\rho^p \leq \rho^p.
\]
Hence, $(x,t) \in \mathcal{B}_{\rho^p}(x_0, - \rho^p)$ and \eqref{eq:bmu2} follows.

Recall that $u(z_1,t_1)\leq L_0m_0$.  We define the function $w :Q_3^- \to \R$,
\[
w(x,t)=\frac{u(\tfrac{\rho}{32} x+z_1,(\tfrac{\rho}{32})^p (L_0m_0)^{-(p-2)}t +t_1 )}{L_0m_0}. 
\]
By Remark \ref{rem:scaling}  $w$ is a nonnegative supersolution of \eqref{eq1} and $w(0,0) \le 1$.
We apply Lemma~\ref{lem:ABP} to conclude 
 \[
|\{(x,t) \in Q_1^- \, : \, w(x,t) < 4\}| \geq c_0.
\]
Going back to the original coordinates and using $Q_{\rho/32}^- ((L_0m_0)^{-(p-2)}) \subset Q_{\rho/32}^-$ we find 
\begin{equation}\label{eq:bmu4}
\big|\{(x,t) \in (z_1,t_1)+Q_{\rho/32}^- \, : \, u(x,t) < 4L_0m_0\}\big| \geq c_0 \left(\frac{\rho}{32}\right)^{n+p}(L_0m_0)^{-(p-2)}.
\end{equation}

Let us next show that  every point $(x,t)$ in $ (z_1,t_1)+Q_{\rho/32}^-$ satisfies $x \in B_1$.  Recall that $|z| \le 1-\tfrac{\rho}{8}$ and $y_m \in B_{1/32}(\tfrac{2}{\rho}(z-x_m)).$ These imply
\[
|z_1| =\left|\frac{\rho}{2} y_m + x_m + z-z \right| \le \frac{\rho}{2}\left| y_m-\frac{2}{\rho}(z-x_m)\right| +|z| \le \frac{\rho}{64} +1-\frac{\rho}{8} \le 1-\frac{\rho}{16}.
\] 
Consequently, for $(x,t) \in (z_1,t_1)+Q_{\rho/32}^-$, we have  $|x| < |z_1|+\frac{\rho}{32} \le 1-\frac\rho{32} <1$. Hence, we deduce from \eqref{eq:bmu2} and  \eqref{eq:bmu4}  that 
 \[
\big|\{(x,t) \in \B_{\rho^p}(x_0,-\rho^p) \, : \, x \in B_1 \quad \text{and}\quad  u(x,t) < 4L_0m_0\}\big| \geq c_0 \left(\frac{\rho}{32}\right)^{n+p}(L_0m_0)^{-(p-2)}
\] 
 and the claim follows. 
\end{proof}

We  need the following simple modification of the Vitali covering argument. We give its proof in the Appendix for the reader's convinience. 

\begin{prop}\label{prop:vitali}
Let $\theta \in (0,1)$ and assume $\mathcal{F}= \{ Q_i\}_{i}$ is a family of homothetic cylinders such that 
\[
Q_i = (x_i,t_i)  +  \bar Q_{\rho_i}^+(\theta)
\]
for $(x_i,t_i) \in \R^{n+1}$ and $\rho_i \in (0,1]$. Then there is a countable  subcover  $ \tilde \F= \{ Q_k\}_{k }$ such that $Q_k \in \tilde \F$ are disjoint and
\[
\bigcup_{Q_i \in \mathcal{F}}  Q_i  \subset  \bigcup_{Q_k \in  \tilde \F}  5 Q_k.
\]
Here $5 Q_k := (x_i,t_i)  +  \bar Q_{5\rho_i}(\theta)$.
\end{prop}

\begin{proof}[\textbf{Proof of Theorem \ref{thm:weak-harnack}}]
Let $m_0 >0$ be as in Lemma  \ref{lem:barrier} and set $c = m_0$.  We note that by translating and scaling  the coordinates  we may assume that $(x_0, t_0) = (0,0)$ and $\rho =1$. Moreover, in order to have the intrinsic cylinder $Q$ of size one, we scale the function such that  $u(0,0) =m_0$. Indeed, the latter follows by considering the function 
\[
v(x,t) = \frac{u(x,(m/m_0)^{-(p-2)} t)}{m/m_0},
\]
 where $m = u(0,0)$.

 In view of the layer-cake formula, it is enough to prove that there are universal constants $C>1$, $\eta \in (0,1)$ and $k_1 \in  \mathbb{N}$ such that  
\begin{equation}
\label{eq:weak-h-1}
|\{(x,t)\in B_1 \!\times \! (-2,-1] : \, u(x,t) > L^km_0\}| \leq C \eta^k, \qquad \text{for all }\, k \geq k_1,
\end{equation}
where $L=4L_0$ and $L_0>1$ is from Lemma \ref{lem:barrier}. 

Choose $\gamma>1$ such that 
\begin{align}\label{eq:wh-gamma}
\gamma L^{-(p-2)}+1 \le \gamma 
\end{align}
and then  $k_1 \in  \mathbb{N}$ such that $\gamma L^{-k_1(p-2)} \leq \frac18$. We define shrinking cylinders $Q_{k+1} \subset Q_k$ as  
\[
Q_k=B_1 \!\times\! (-2,-1+ \gamma L^{-k(p-2)}]
\]
and notice that  $B_1 \times\! (-2,-1]  \subset Q_k  \subset B_1 \times\! (-2,-\tfrac12]$
 for all $k \geq k_1$. For every $k\geq k_1$ we define the sublevel set 
\[
A_k=\{(x,t)\in Q_k : \, u(x,t) \leq L^km_0\} 
\]
and the associated intrinsic scaling
\[
a_k := L^{-k(p-2)}.
\]

 For every point $(x_0,t_0) \in Q_{k+1} \setminus A_k$, we define $r>0$ as the smallest number for which the intersection $\B^{a_k}_r(x_0,t_0)\cap A_k$ is non-empty, where $\B^{a_k}_r(x_0,t_0)$ is defined in \eqref{def:B-a}. We will later show that thanks to  Lemma \ref{lem:barrier}  such a radius $r$ exists.   We define the cylinder $C_{(x_0,t_0)}$ as the smallest cylinder that contains  $\B^{a_k}_r(x_0,t_0)$. It is easy to see that by defining the radius 
\begin{equation}\label{def:radius-k}
\rho_k = \rho_k(x_0,t_0) := \left( \frac{r}{a_k}\right)^{1/p},
\end{equation}
we may write $C_{(x_0,t_0)}$ as 
\[
C_{(x_0,t_0)} =  (x_0,t_0)  +  \bar Q_{\rho_k}^+(a_k),
\]
where $Q_\rho^+(a_k)$ is the homothetic cylinder defined in \eqref{def:cylinders1}. This means that the interior of the set $\B^{a_k}_r(x_0,t_0)$ is contained in $\R^{n+1} \setminus A_k$ and there is a point of intersection  $(\hat x, \hat t)  \in \B^{a_k}_r(x_0,t_0) \cap A_k$ with $\hat t = t_0 +r$. In other words, we have 
\begin{align}\label{eq:wh_inf}
\inf_{B_{\rho_k}(x_0) \cap B_1}u(\cdot, t_0+a_k \rho_k^p) \le L^k m_0.
\end{align} 
In this way, we obtain a family of cylinders $\mathcal{F} = \{C_{(x,t)}\}_{(x,t) \in Q_{k+1} \setminus A_k } $ which  is a cover of  the set $Q_{k+1} \setminus A_k $.

Let  $(x_0,t_0) \in Q_{k+1} \setminus A_k$, and let $C_{(x_0,t_0)}$ and $\rho_k$ be as above.  Recall that $\rho_k$, defined in \eqref{def:radius-k}, is the spatial width of the cylinder $C_{(x_0,t_0)}$ and it crucial that we  bound it from above. To this aim,   notice first   that by the definitions of $Q_{k+1}$ and $a_k$, and by \eqref{eq:wh-gamma}
\begin{equation}\label{eq:weakharn-2}
\begin{split}
  t_0+ a_k &\le -1+\gamma L^{-(k+1)(p-2)} +L^{-k(p-2)}  \\
&=-1 + L^{-k(p-2)}(\gamma L^{-(p-2)} +1)  \le -1 + \gamma L^{-k(p-2)}  .
\end{split}
\end{equation}
We choose $(x_1,t_1)$ with $x_1 = \frac{x_0}{2}$ and $t_1 = t_0+ a_k$. By \eqref{eq:weakharn-2} and by the choice of $k_1$  we have $t_1 \leq  -\tfrac{1}{2}$, and from $x_0 \in B_1$, it follows  $|x_1| <\tfrac12$. 
Since $u(0,0)\leq m_0$, we may apply Lemma  \ref{lem:barrier} to find a point   $\bar x \in B_{1/32}(x_1)$ such that 
\[
u(\bar x, t_1) \leq L_0m_0.
\]  
We use $\bar x \in B_{1/32}(x_1)$,  the choice of $x_1$, and the fact that $x_0 \in B_1$, to estimate
\[
|\bar x - x_0| \leq |\bar x - x_1 | + |x_1 - x_0| < \frac{1}{32} + \big| \frac{x_0}{2} - x_0 \big| < \frac{1}{32} + \frac12 <1, 
\]
which implies  $\bar x \in B_1(x_0)\cap B_1$.  Since $u(\bar x, t_1) \leq L_0m_0$,  it  holds, by recalling that $t_1 = t_0+ a_k$  and using \eqref{eq:weakharn-2}, that $(\bar x, t_1) \in A_k$. Therefore, we conclude that the cylinder $(x_0,t_0)  +  \bar Q_{1}^+(a_k)$
intersects the set $A_{k}$. Hence, we have
\begin{equation}\label{eq:radius-bounded}
 \rho_k  \le  1. 
\end{equation}

Let  $(x_0,t_0) \in Q_{k+1} \setminus A_k$  and consider the cylinder $C_{(x_0,t_0)} $ constructed above. We claim that there is a universal constant $c_2>0$ such that 
\begin{equation}\label{eq:good-measure}
| \tilde A_{k+1} \cap C_{(x_0,t_0)}  \setminus A_k| \geq c_2 |C_{(x_0,t_0)} |,
\end{equation}
where
\[
\tilde A_{k+1}=\{(x,t)\in Q_k : \, u(x,t) \leq L^{k+1}m_0\}.
\]
To this aim we consider the function $v:Q_3^- \to \R$, 
\[
v(y,\tau)= \frac{ u(y,a_k \tau+t_0+a_k\rho_k^p)}{L^k}.
\]
Then, by Remark \ref{rem:scaling}, $v$ is a nonnegative supersolution of \eqref{eq1} and by \eqref{eq:wh_inf}, we have 
\[
\inf_{B_{\rho_k}(x_0) \cap B_1} v(\cdot, 0) \le m_0.
\]
Moreover, by \eqref{eq:radius-bounded}, it holds $\rho_k \le 1$. In conclusion,  we may apply Lemma \ref{prop:bmu} to get
\[
\left| \left\{ (y, \tau) \in \mathcal{B}_{\rho_k^p}(x_0, - \rho_k^p) \, : \, y \in B_1 \ \text{ and } \ v(y, \tau) < 4 L_0m_0 \right\} \right| \geq c_1 \rho_k^{n + p},
\]
where $\mathcal{B}_{\rho_k^p}(x_0, - \rho_k^p)$ is defined in \eqref{def:B-a}.  Writing this in the  original coordinates and recalling that $4L_0=L$, we find
\begin{equation}\label{eq:weakharn-1}
\left| \left\{ (x, t) \in \mathcal{B}^{a_k}_r(x_0, t_0) \, : \, x \in B_1 \ \text{ and } \ u(x, t) <  L^{k+1}m_0 \right\} \right| \geq c_1 a_k\rho_k^{n + p}.
\end{equation}

Recall that $(x_0,t_0) \in Q_{k+1}$ and we constructed the paraboloid $\mathcal{B}^{a_k}_r(x_0, t_0)$ such that its interior   is contained in $\R^{n+1} \setminus A_k$. 
Let us finally show that for every $(x,t) \in \mathcal{B}^{a_k}_r(x_0, t_0)$ with $x \in B_1$ it holds $(x,t) \in Q_k$. We need only to show that $t \leq -1 + \gamma L^{-k(p-2)}$.  We trivially have 
$t-t_0 \leq r$, and  \eqref{def:radius-k} and  \eqref{eq:radius-bounded} imply $r \leq a_k$. We then obtain from \eqref{eq:weakharn-2} that 
\[
t = t_0 + (t - t_0) \leq t_0 + a_k \leq  -1 + \gamma L^{-k(p-2)}  .
\] 
Hence, $(x,t) \in Q_k$. We thus conclude from \eqref{eq:weakharn-1} that 
\[
\left|  \mathcal{B}^{a_k}_r(x_0, t_0) \cap \tilde A_{k+1} \setminus A_k  \right| \geq c_1 a_k\rho_k^{n + p}.
\]
Since $\mathcal{B}^{a_k}_r(x_0, t_0) \subset C_{(x_0,t_0)}$ and $|C_{(x_0,t_0)}|=\omega_n a_k\rho_k^{n + p}$, where $\omega_n$ is the measure of $n$-dimensional unit ball,  we trivially have
\[
\left| C_{(x_0,t_0)} \cap \tilde A_{k+1} \setminus A_k \right| \geq \frac{c_1}{\omega_n}  |C_{(x_0,t_0)}|.
\]
This completes the proof of   \eqref{eq:good-measure}.

We now prove \eqref{eq:weak-h-1} from \eqref{eq:good-measure}. Recall that  our family of cylinders  $\mathcal{F} = \{C_{(x,t)}\}_{(x,t) \in Q_{k+1} \setminus A_k } $  is a cover of  $ Q_{k+1} \setminus A_k $. We use the version of the Vitali covering theorem stated in Proposition \ref{prop:vitali} to obtain a countable  subcover $\mathcal{F}_1 = \{C_{(x_i,t_i)}\}_{i} $, where the cylinders are pairwise disjoint and $\{5 \, C_{(x_i,t_i)}\}_{i} $ still cover $Q_{k+1} \setminus A_k $. Therefore, by using \eqref{eq:good-measure} we have that   
\[
\begin{split}
|Q_{k+1} \setminus A_k | & \leq \sum_{i} |5 \, C_{(x_i,t_i)} |\leq C \sum_{i} |C_{(x_i,t_i)}| \\
&\leq C  \sum_{i} |\tilde A_{k+1} \cap  C_{(x_i,t_i)}  \setminus A_k| \leq C |A_{k+1}   \setminus A_k| +C L^{-k(p-2)},
\end{split}
\]
where the last  inequality follows from the fact that the cylinders are disjoint and $|\tilde A_{k+1} \setminus  A_{k+1}| \le C L^{-k(p-2)}.$ Thus we have
\[
\begin{split}
|Q_{k+1} \setminus A_{k+1} | &=  |Q_{k+1} \setminus A_{k}| - |A_{k+1} \setminus A_{k}|\\
& \leq \left(1 - \frac{1}{C}\right) |Q_{k+1} \setminus A_k |+ C L^{-k(p-2)}\\
& \le \left(1 - \frac{1}{C}\right) |Q_{k} \setminus A_k | +C L^{-k(p-2)},
\end{split}
\]
where in the last inequality, we have used $Q_{k+1} \subset Q_k.$
This gives \eqref{eq:weak-h-1} by iteration argument and  completes the proof of the theorem.

\end{proof}
We now write the following corollary of the weak  Harnack inequality. 
\begin{cor}\label{Cor:bmu1}
Let $0<m_0<1$ be as in Lemma \ref{lem:barrier} and assume that $u$ is  a nonnegative viscosity supersolution  of \eqref{eq1} in $Q_3^-.$ There exists a universal constant $L_1$ such that if $u(0,0) \le m_0$, then  
\begin{equation}\label{eq:super-estimate}
\big|\{ (x,t) \in  B_{1} \!\times\! (-2, -1] :  u(x,t) \geq L_1 m_0 \}\big|  \leq  \frac{\omega_n}{4^{n+1}},
\end{equation}
where $\omega_n$ is $n$-dimensional Lebesgue measure of $B_1\subset \R^n$.
\end{cor} 
\begin{proof}
The estimate  \eqref{eq:weak-h-1} in the proof of Theorem \ref{thm:weak-harnack} reads as
\begin{equation}
\big|\{ (x,t) \in  B_{1}\! \times\!(-2, -1] :  u(x,t) \geq L^k m_0 \}\big|  \leq C \eta^k.
\end{equation}
Choose $k_1$ large enough such that $C\eta^{k_1} \le \frac{\omega_n}{4^{n+1}}$. The claim follows by choosing $ L_1 = L^{k_1}.$
\end{proof}

\section{Harnack inequality}

By coupling the propagation of boundedness from  Lemma \ref{lem:propagation-space-3}  with the measure estimate from Corollary \ref{Cor:bmu1}, we are able to prove the Harnack inequality by adapting the  argument from \cite{cc, W1}.  
\begin{lemma}
\label{prop:subsolution}
Let $0 < m_0 < 1$ and $L_1 > 1$ be as in Lemma~\ref{lem:barrier} and Corollary~\ref{Cor:bmu1}, respectively. Assume that $u: Q_3^- \to \mathbb{R}$ is a nonnegative viscosity supersolution of~\eqref{eq1} and  subsolution of~\eqref{eq2}, with $u(0,0) \leq m_0$. Then there exists a universal constant $k_1 \in \mathbb{N}$ with the following property: if $(x_0, t_0) \in B_{4/3} \!\times\! (-3, -1]$ and
\begin{align}\label{eq:sub-Asum}
u(x_0, t_0) \geq 2 \nu^{k-1} L_1 m_0,
\end{align}
for $\nu = \frac{L_1}{L_1 - 1/2}$ and $k \geq k_1$, then
\[
\sup_{(x_0, t_0)+Q^-_{3\rho_0^k}(b_k)} u(x, t) \geq 2 \nu^{k} L_1 m_0,
\]
where $b_k = \nu^{-k(p-2)}$ and $\rho_0 = \rho_0(\nu) \in (0,1)$ is as in Lemma~\ref{lem:propagation-space-3}.

 \end{lemma}
  
\begin{proof}
Choose $k_1 > k_0$, where $k_0$ is from Lemma~\ref{lem:propagation-space-3},   such that $\rho_0^{k_1} \le \tfrac{1}{4}$.  
We argue by contradiction and assume that \eqref{eq:sub-Asum}
 holds but 
 \begin{align}\label{eq:sub-cont}
\sup_{(x_0,t_0)+Q_{3 \rho_0^k}^-(b_k)} u(x,t) < 2\nu^{k} L_1 m_0.
\end{align}
Consider the  function $v : Q_3^- \to \mathbb{R}$,
\[
v(y,s) = \frac{2  \nu^{k} L_1 m_0 - u(\rho_0^k y + x_0, \rho_0^{k p} b_k s + t_0)}{2 \nu^{k-1} (\nu - 1) L_1},
\]
Then, by the choice of $k_1$ and the assumption $k \ge k_1$, the function $v$ is defined in $Q_3^-$. Since $b_k = \nu^{-k(p-2)} = (2 \nu^{k-1} (\nu - 1) L_1)^{-(p-2)}$, it follows from Remark~\ref{rem:scaling} and~\eqref{eq:sub-cont} that $v$ is a nonnegative supersolution of~\eqref{eq1}. Moreover, using the assumption \eqref{eq:sub-Asum}, we have $v(0,0) \leq m_0$.
 We may apply Corollary~\ref{Cor:bmu1} to deduce that
\[
\left| \left\{ (y,s) \in B_1\! \times \!(-2, -1] : v(y,s) \geq L_1 m_0 \right\} \right| \leq \frac{\omega_n}{4^{n+1}}.
\]
Writing in the original coordinates and using $2(\nu-1)L_1=\nu$, we find
\begin{equation}\label{eq:super-estimate-2}
\big|\{ (x,t) \in  (x_0,t_0) + B_{\rho_0^{k}} \! \times \! I_k \; : \;  u(x,t) \leq \nu^{k}L_1 m_0   \}\big|  \leq \frac{\omega_n}{4^{n+1}} b_k \rho_0^{k(n+p)},
\end{equation}
where $I_k = (-2 \rho_0^{kp}b_k,  -\rho_0^{kp}b_k]$. 

By assumption, $(x_0,t_0) \in B_{4/3} \! \times \! (-3, -1]$ and $u(0,0) \le m_0$. Therefore by Lemma~\ref{lem:propagation-space-3} there is  a point $\bar{x} \in B_{\rho_0^{k}}(x_0)$ such that
\[
u\left(\bar{x},\, t_0 \right) \leq  \nu^{k} m_0.
\]
We now define  $w:  Q_3^- \to \R $ as
\[
w(y,s) = \frac{u(\rho_0^{k} y +\bar x, \rho_0^{kp} b_k \, s + t_0)}{\nu^{k} } .
\]
Since $b_k=\nu^{-k(p-2)}$, it follows from Remark \ref{rem:scaling} that  $w$ is a nonnegative supersolution of \eqref{eq1} with $w(0,0) \le m_0.$ We may apply Corollary \ref{Cor:bmu1} to $w$ and obtain
\[
\big|\{ (y,s) \in  B_{1} \!\times \! (-2, -1] :  w(y,s) \geq L_1 m_0 \}\big|  \leq \frac{\omega_n}{4^{n+1}}.
\]
Rewriting in the original coordinates, we find
\begin{equation}\label{eq:super-estimate-3}
\big|\{ (x,t) \in  (\bar x,t_0) +  B_{\rho_0^{k}} \! \times \! I_k  :  u(x,t) \geq  \nu^{k}L_1 m_0  \}\big|  \leq \frac{\omega_n}{4^{n+1}} b_k \rho_0^{k(n+p)}.
\end{equation}

It is easy to see that $\bar{x} \in B_{\rho_0^{k}}(x_0)$ implies 
\begin{equation}\label{eq:super-estimate-4}
B_{\rho_0^{k}/4}\left( \tfrac{x_0 + \bar{x}}{2} \right)  \subset  B_{\rho_0^{k}}(\bar{x}) \cap B_{\rho_0^{k}}(x_0).
\end{equation}
Using \eqref{eq:super-estimate-4} and \eqref{eq:super-estimate-2}, we trivially obtain
\[
\big|\{ (x,t) \in  (\tfrac{x_0 + \bar{x}}{2},t_0)+ B_{\rho_0^{k}/4}\! \times \! I_k \; : \;  u(x,t) \leq   \nu^{k}L_1 m_0 \}\big|  \leq \frac{\omega_n}{4^{n+1}} b_k \rho_0^{k(n+p)},
\] 
while  \eqref{eq:super-estimate-4} and \eqref{eq:super-estimate-3} imply 
\[
\big|\{ (x,t) \in  (\tfrac{x_0 + \bar{x}}{2},t_0)+ B_{\rho_0^{k}/4}\! \times \! I_k \; : \;  u(x,t) \geq   \nu^{k}L_1 m_0 \}\big|  \leq \frac{\omega_n}{4^{n+1}} b_k \rho_0^{k(n+p)}. 
\]
Combining the above two estimates, we deduce that
\[
\frac{\omega_n}{4^{n}}\, b_k \rho_0^{k(n+p)}=\left|B_{\rho_0^{k}/4}(\tfrac{x_0 + \bar{x}}{2})\! \times \! (-2 \rho_0^{kp}b_k,  -\rho_0^{kp}b_k] \right| \le  \frac{2\omega_n}{4^{n+1}} b_k \rho_0^{k(n+p)},
\]
which is a contradiction. This completes the proof of the lemma.

\end{proof}
%
%
%
\begin{proof}[\textbf{Proof of Theorem \ref{thm:harnack}}]
Let $m_0 >0$ be as in Lemma  \ref{lem:barrier} and set $c_1 = m_0$. Arguing as in the beginning of the proof of Theorem  \ref{thm:weak-harnack}, we may assume that  $(x_0, t_0) = (0, 0)$, $\rho = 1$, and $u(0, 0) = m_0$, in which case $\theta_1 = 1$. Hence, it suffices to show that there exists a universal constant $C_1$ such that
\begin{align}\label{eq:backward_Harnack}
 \sup_{Q_1^-} u(\cdot, \cdot -1) \leq C_1 m_0.
\end{align}

Let $L_1$, $k_1$, $\nu$ and $\rho_0 \in (0, 1)$ be as in Lemma \ref{prop:subsolution}. Choose $k_2 > k_1$ such that 
\begin{align}\label{eq:har-1}
3^p\sum_{k = k_2}^{\infty} \rho_0^k < \frac{1}{8}.
\end{align}
We claim that \eqref{eq:backward_Harnack} holds with $C_1=2 \nu^{k_2} L_1$, i.e.,
\begin{align}\label{eq:har-2}
\sup_{Q_1^-} u(\cdot, \cdot -1)  \leq 2 \nu^{k_2} L_1 m_0.
\end{align}

We argue by contradiction and assume  \eqref{eq:har-2} is not true. Then, there exists $(x_1,t_1) \in Q_1^-$ such that 
\[
u(x_1, t_1-1) > 2 \nu^{k_2} L_1 m_0.
\]
Since $(x_1,t_1-1) \in B_{4/3} \! \times \!(-3,-1]$ and $k_2 > k_1$, we may apply Lemma \ref{prop:subsolution} to deduce that there exists a point $(x_2, t_2)$ such that  
\[
|x_2 - x_1| \leq 3 \rho_0^{k_2 + 1}, \quad t_1 - 3^p \rho_0^{(k_2 + 1)p} \le t_2 \leq t_1,
\]
such that
\[
u(x_2, t_2-1) \geq 2 \nu^{k_2 + 1} L_1 m_0.
\]
Proceeding in a similar fashion, we find a sequence $(x_l, t_l)$ with $l \geq 2$ such that  
\[
|x_{l} - x_{l-1}| < 3 \rho_0^{k_2 + l-1}\quad \text{and} \quad   t_{l-1} -3^p \rho_0^{p(k_2 +l-1 )} \le t_{l} \le t_{l-1},
\]
such that
\begin{align}\label{eq:har-3}
u(x_l, t_l-1) \geq 2 \nu^{k_2 + l - 1} L_1 m_0.
\end{align}
This is possible since, from \eqref{eq:har-1}, for all $l \ge 2$ we have  
\[
|x_l| \le |x_1| + 3 \sum_{j=1}^{l-1} \rho_0^{j+k_2} < 1+\frac18 < \frac43
\]
and 
\[
   -1 -\frac18 < t_1 -3^p\sum_{j=1}^{l-1} \rho_0^{(j+k_1)p} \le t_l \le t_1 \le 0.
\] 
Because of these bounds,  each point $(x_l,t_l-1)$ stays inside $B_{4/3} \!\times \! (-3,-1]$ allowing us to apply Lemma \ref{prop:subsolution} repeatedly. Therefore the sequence $(x_l,t_l-1)$ converges to  $(x_\infty, t_\infty) \in B_{4/3} \!\times\! (-3, -1]$. Moreover, by the continuity of $u$, we have $u(x_l, t_l-1) \to u(x_\infty, t_\infty) < \infty$, which contradicts \eqref{eq:har-3}. This completes the proof of \eqref{eq:har-2}, which in turn implies  \eqref{eq:backward_Harnack}.

Our next goal is to prove the forward-in-time Harnack inequality. Let $C_1$ be as in \eqref{eq:backward_Harnack}, and set $c_2=2C_1m_0$. Recall that we assume $(x_0,t_0)=(0,0)$, $\rho=1$ and $u(0,0)=m_0$, and therefore $\theta_2 = (2C_1)^{p-2}$. By the assumption $u$ is a nonnegative viscosity solution of \eqref{maineq} in the cylinder $(0, 2\theta_2) + Q_{4}^-(\theta_2)$.  Therefore, it is enough to prove 
\begin{equation}\label{eq:forward_Harnack}
  m_0=u(0,0) \le 2 C_1\inf_{Q_1^+(\theta_2)} u(\cdot, \cdot+\theta_2).
\end{equation}

The proof is via continuity type argument. Since $u$ is continuous, for small  $r>0$ it holds  
\[
m_0=u(0,0)< 2C_1 \inf_{Q_r^+(\theta_2)}u(\cdot ,\cdot + \theta_2 r^p).
\] 
We choose the  smallest $r \le 1$, denoted by $r_s$, for which 
\begin{equation}\label{eq:cont1}
m_0=u(0,0)= 2C_1 \inf_{Q_{r_s}^+(\theta_2)} u(\cdot,\cdot + \theta_2r_s^p).
\end{equation}
The claim follows once we show    $r_s = 1$.  We argue by contradiction and  assume that  $r_s<1$. From \eqref{eq:cont1} we deduce that  there exists a point $(x_s,t_s)$  in the closure of $Q_{r_s}^+(\theta_2)$ such that  
\[
m_0= 2C_1 u(x_s,t_s+ \theta_2 r_s^p).
\]
We  define $v:Q_3^- \to \R$,
\[
v(x,t)=2C_1u(r_sx+x_s,\theta_2 \, r_s^p \, (t+1)+t_s).
\]
Notice that since $u$ is defined in $(0,2\theta_2)+ Q_{4}^-(\theta_2)$, $v$ is  defined in $Q_3^-.$ Indeed, first $x \in B_3$ and $x_s \in \bar B_{r_s}$  imply 
$ |r_sx+x_s| < 4$. Regarding the time variable, since $t\leq 0$ and  $t_s < \theta_2$, we easily have $\theta_2  r_s^p  (t+1)+t_s < 2 \theta_2.$ Using $t >-3^p$ and $t_s \ge 0$, we find
\[
\theta_2  r_s^p  (t+1)+t_s \ge  \theta_2 (-3^p+1) >  -4^p \theta_2  +2\theta_2.
\]

From the definition of $\theta_2=(2C_1)^{p-2}$, we have that $v$ is a nonnegative viscosity supersolution of \eqref{eq1} and viscosity subsolution of \eqref{eq2} in $Q_3^-$ with $v(0,0)= 2C_1 u(x_s,t_s+\theta_2 r_s^p)=m_0.$ Thus, we may apply \eqref{eq:backward_Harnack} to obtain
\[
\sup_{Q_1^-} v(\cdot, \cdot - 1) \le C_1 v(0,0).
\]
Since the point $\left(-\tfrac{x_s}{r_s}, -\tfrac{t_s}{\theta_2 r_s^p}\right)$ lies in the closure of $Q_1^-$, the above inequality, in particular, yields
\[
2C_1 u(0,0) \le C_1 m_0,
\]
which is a contradiction, since $u(0,0) = m_0 > 0$.
 Hence, we have $r_s=1$ and 
\[
m_0=u(0,0)\le  2C_1 \inf_{Q_1^+(\theta_2)} u(\cdot,\cdot +\theta_2).
\]
Finally, we take $C=2C_1.$ This completes the proof of the theorem.

\end{proof}

\section*{Appendix}

\begin{proof}[\textbf{Proof of Proposition \ref{prop:vitali}}]
Without loss of generality, using the fact that $\R^{n+1}$ is a Lindelöf space, we may  assume that $\F$ is a countable family. We begin by partitioning the family  $\mathcal{F}$ as follows: for $k \ge 0$,
\[\mathcal{F}_k=\{Q_i \in \mathcal{F} \ : \ 2^{-k-1} < \rho_i \le 2^{-k}\},
 \]
where $Q_i= (x_i,t_i)  +  \bar Q_{\rho_i}^+(\theta)$. We first choose  $\tilde \F_1$ to be a maximal family of disjoint cubes in $\F_1$. By  Zorn's lemma, this set can be chosen. We define a subfamily $\tilde \F_k \subset \F_k$ for $k >1$ as  follows: Assume $\tilde \F_j$ has been defined for $1 \le j \le k-1.$ First, we define the set 
\[
\mathcal H_k= \{Q_i \in \F_k \ :\ Q_i \cap Q = \emptyset \,\, \text{for all} \,\,  Q \in \cup_{j=1}^{k-1}\tilde \F_j\}
\]
We choose $\tilde \F_k$ to be a maximal family of disjoint cubes in $\mathcal H_k$. We claim that $\cup_{k=1}^{\infty}\tilde \F_k$ is the required set, $\tilde \F.$ Clearly, cubes in $\tilde \F$ are pairwise disjoint. Let $Q_i \in \F_k$ for some $k\ge 0.$ We need to show $Q_i \subset  \bigcup_{Q_k \in  \tilde \F}  5 Q_k$.  One of the following holds:
\begin{itemize}
\item[(i)] $Q_i \in \tilde \F_k,$ in  which case the claim is trivially true. 
\item[(ii)] $Q_i \in \mathcal H_k \setminus \tilde \F_k$, and then by the maximality of  $\tilde \F_k,$ we have that $Q_i \cap Q \neq \emptyset$ for some $Q \in \tilde \F_k.$
\item[(iii)]$Q_i \notin \mathcal H_k,$ then, we have that $Q_i \cap Q \neq \emptyset$ for some $Q \in \cup_{j=1}^{k-1}\tilde \F_j.$
\end{itemize}
In  the cases (ii) and (iii), we have that $Q_i \cap Q_l \neq \emptyset$ for some $Q_l \in \cup_{j=1}^{k}\tilde \F_j.$ Since $Q_i \in \F_k,$ $2^{-k-1} < \rho_i \le 2^{-k}.$ Also, $Q_l \in \cup_{j=1}^{k}\tilde \F_j$ gives us that  $2^{-k-1} < \rho_l \le 1$, which implies $\rho_i \leq 2 \rho_l$.  We  show that $Q_i \subset 5Q_l.$ Let $(x,t) \in Q_i$ and $(y,s) \in Q_i \cap Q_l.$ Notice that 
\[
|x-x_l| \le |x-y|+|y-x_l| \le 2\rho_i + \rho_l  \leq  4 \rho_l +\rho_l=5 \rho_l
\]
and
\[
|t -t_l| \le |t-s|+|s-t_l| \le \theta \rho_i^p + \theta \rho_l^p \le \theta (2\rho_l)^p + \theta \rho_l^p  \le \theta (5\rho_l)^p. 
\]
This completes the proof.
\end{proof}

\section*{Acknowledgments}
The authors were supported by the Academy of Finland grant 314227.

\end{document}